\pdfoutput=1
\documentclass[a4paper]{amsart}

\usepackage[utf8]{inputenc}
\usepackage[T1]{fontenc}
\usepackage[safeinputenc=true,style=alphabetic,firstinits,useprefix=true,maxnames=5,doi=false,eprint=true,isbn=false,url=false]{biblatex}
\newbibmacro{string+doiurl}[1]{%
  \iffieldundef{doi}{%
    \iffieldundef{url}{%
         #1%
      }{%
      \href{\thefield{url}}{#1}%
    }%
  }{%
    \href{http://dx.doi.org/\thefield{doi}}{#1}%
  }%
}
\DeclareFieldFormat[article,misc]{title}{\usebibmacro{string+doiurl}{\mkbibquote{#1}}}

\usepackage{amsmath,amsthm,amssymb}
\numberwithin{equation}{section}
\newtheorem{theorem}[equation]{Theorem}

\newtheorem{lemma}[equation]{Lemma}
\newtheorem{corollary}[equation]{Corollary}
\newtheorem{claim}[equation]{Claim}
\theoremstyle{definition}
\newtheorem{remark}[equation]{Remark}

\usepackage{mathtools}
\DeclarePairedDelimiter\abs{\lvert}{\rvert}
\DeclarePairedDelimiter\meas{\lvert}{\rvert}
\DeclarePairedDelimiter\norm{\lVert}{\rVert}
\DeclarePairedDelimiter\Set\{\}
\DeclarePairedDelimiterX\innerp[2]{\langle}{\rangle}{#1,#2}
\usepackage{tikz,pgfplots}
\usepackage[charter]{mathdesign}
\usepackage{hyperref}
\hypersetup{
bookmarks=true,
colorlinks=true,
citecolor=blue,
pdfpagemode=UseNone,
pdfstartview=FitH,
pdfdisplaydoctitle=true,
pdftitle={Square functions for bi-Lipschitz maps and directional operators},
pdfauthor={di Plinio, Guo, Thiele, Zorin-Kranich},
pdflang=en-US
}

\newcommand{\Z}{\mathbb{Z}}

\newcommand{\C}{\mathbb{C}}
\newcommand{\N}{\mathbb{N}}
\newcommand{\R}{\mathbb{R}}
\newcommand{\dif}{\mathrm{d}}
\renewcommand{\C}{\mathbb{C}}
\newcommand{\calI}{\mathcal{I}}
\newcommand{\calP}{\mathcal{P}}
\newcommand{\calQ}{\mathcal{Q}}
\newcommand{\calR}{\mathcal{R}}

\newcommand{\calG}{\mathcal{G}}

\newcommand{\calC}{\mathcal{C}}

\newcommand{\bfE}{\mathbf{E}}
\newcommand{\one}{\mathbf{1}}
\newcommand{\Lip}{\mathrm{Lip}}
\DeclareMathOperator{\supp}{supp}
\DeclareMathOperator{\dist}{dist}
\newcommand{\Aem}{\mathcal{A}}
\newcommand{\Dem}{\mathcal{D}}
\newcommand{\Dini}{\mathrm{Dini}}

\title{Square functions for bi-Lipschitz maps and directional operators}
\author{Francesco Di Plinio}
\address[FDP]{Department of Mathematics, University of Virginia, Kerchof Hall, Box 400137, Charlottesville, VA 22904-4137, USA}
\email{francesco.diplinio@virginia.edu}
\author{Shaoming Guo}
\address[SG]{Indiana University Bloomington, 831 E Third St, Bloomington, IN 47405, USA}
\email{shaoguo@iu.edu}
\author{Christoph Thiele}
\author{Pavel Zorin-Kranich}
\address[CT, PZ]{Mathematical Institute, University of Bonn, Endenicher Allee 60, 53115 Bonn, Germany}
\email{thiele@math.uni-bonn.de}
\email{pzorin@math.uni-bonn.de}

\subjclass[2010]{42B25}

\begin{document}

\begin{abstract}
First we prove a Littlewood-Paley diagonalization result for bi-Lipschitz perturbations of the identity map on the real line.
This result entails a number of corollaries for the Hilbert transform along lines and monomial curves in the plane.
Second, we prove a square function bound for a single scale directional operator.
As a corollary we give a new proof of part of a theorem of Katz on direction fields with finitely many directions.
\end{abstract}
\maketitle

\section{Introduction}

This paper grew out of a study of variable directional operators in the plane.
We present two main results together with some corollaries.

It is a folklore conjecture and discussed by several authors, for example \cite{MR1232192}, \cite{MR2654385}, \cite{MR3592519}, that Lipschitz is the critical regularity assumption on a direction field to yield $L^p$ boundedness of some associated directional operators.
Possibly at the heart of positive results in this direction appears to be a one dimensional Littlewood-Paley diagonalization estimate for bi-Lipschitz maps, which is our first main theorem.

\begin{theorem}
\label{thm:Lip-LP-diag}
Let $A : \R\to\R$ be a Lipschitz function with $\norm{A}_{\Lip} \leq 1/100$ and consider the change of variable $T_{A}f(x):=f(x+A(x))$.

Let $\psi$ be a Schwartz function on $\R$ such that $\widehat{\psi}$ identically equals $1$ on $\pm [99/100,103/100]$ and vanishes outside of $\pm [98/100,104/100]$.
Let $\Psi$ be another Schwartz function on $\R$ such that $\widehat{\Psi}$ is supported on $\pm[1,101/100]$.
Let $P_{t}f:=\psi_{t}*f$ be the Littlewood--Paley operators associated to $\psi$, where $\psi_t(x)=t^{-1}\psi(t^{-1}x)$.
Then
\[
\norm[\Big]{ \sum_{t \in 2^{\Z}} \abs{(1-P_{t})T_{A} (\Psi_{t} * f)} }_{p}
\lesssim_{p,\psi,\Psi}
\norm{A}_{\Lip} \norm{f}_{p},
\quad
1<p<\infty.
\]
\end{theorem}

Note that when the Lipschitz norm of $A$ becomes too large, then in general $T_A$ fails to be a bijection and the estimate of the theorem breaks down.
By rescaling with $c>0$ and a convexity argument the estimate of the theorem remains true for the following expressions in place of the left hand side:
\[
\norm[\Big]{ \sum_{t \in 2^{\Z}} \abs{(1-P_{ct})T_{A} (\Psi_{ct} * f)} }_{p},\quad
\norm[\Big]{ \int_0^\infty \abs{(1-P_{t})T_{A} (\Psi_{t} * f)}\frac{\dif t}t }_{p}.
\]

We call this result a Littlewood-Paley diagonalization result, since it compares for suitable normalization of $\psi$ and $\Psi$
\[
T_Af
=
\int_0^\infty T_{A} (\Psi_{t} * f)\frac{\dif t}t
=
\int_0^\infty \int_0^\infty P_s T_{A} (\Psi_{t} * f)\frac{\dif t}t\frac{\dif s}s
\]
with the diagonal term
\[
\int_0^\infty P_t T_{A} (\Psi_{t} * f)\frac{\dif t}t\ .
\]
The diagonal term by Littlewood-Paley theory and the Fefferman-Stein maximal theorem can and typically will be controlled in $L^p$ norm by that of any of the following square functions
\begin{equation}
\label{eq:sq-fct}
(\int_0^\infty \abs{P_t T_{A} (\Psi_{t} * f)}^2\frac{\dif t}t)^{1/2}\ ,
(\int_0^\infty \abs{M T_{A} (\Psi_{t} * f)}^2\frac{\dif t}t)^{1/2}\ ,
(\int_0^\infty \abs{ T_{A} (\Psi_{t} * f)}^2\frac{\dif t}t)^{1/2},
\end{equation}
where $M$ denotes the Hardy-Littlewood maximal operator.

An application of Theorem \ref{thm:Lip-LP-diag} is to the directional Hilbert transform in the plane defined for measurable $u:\R^2\to [-1,1]$ as
\[
H_u f(x,y):= \operatorname{p.v.}\int_{-1}^{+1}  f(x+r,y+u(x,y)r) \frac{\dif r}r\ .
\]
\begin{corollary}
\label{cor:LP-diag}
Assume that $u(x,\cdot)$ has Lipschitz constant $\leq 1/100$ for almost every $x\in\R$.
With notation as in Theorem \ref{thm:Lip-LP-diag}, we have
\[
\norm[\Big]{ \sum_{t \in 2^{\Z}} \abs{(1-P_{t})H_{u} (\Psi_{t}*_{2} f)} }_{p}
\lesssim_{p,\phi,\Psi}
\norm{f}_{p},
\quad
1<p<\infty.
\]
Here $P_t$ and convolution with $\Psi_t$ act in the second variable.
\end{corollary}
As outlined above, this theorem reduces bounds for $H_u$ to bounds for a square function.
The $L^2$ part of the following corollary is then immediate.

\begin{corollary}
\label{cor:LL-single-band}
Let $u:\R^{2}\to\R$ be such that $u(x,\cdot)$ has Lipschitz constant $\leq 1/100$ for almost every $x\in\R$.
Assume further with notation as in Theorem \ref{thm:Lip-LP-diag} that
\begin{equation}\label{passumption}
\sup_{0<t<t_{0}}\norm{H_{u} (\Psi_{t} *_2 f)}_{p_{0}}
\lesssim
\norm{f}_{p_{0}}
\end{equation}
for some $1<p_{0}\leq 2$ and $t_{0}>0$.
If $p_{0}=2$, then
\begin{equation}\label{2conclusion}
\norm{H_{u} f}_2\lesssim \norm{f}_2.
\end{equation}
If $1<p_{0}<2$, then
\begin{equation}\label{pconclusion}
\norm{H_{u} f}_p\lesssim \norm{f}_p,
\quad 1 + \frac{1}{3-p_{0}} < p < \infty.
\end{equation}
\end{corollary}

Lacey and Li \cite{MR2654385} proved \eqref{passumption} for all $2<p_{0}<\infty$ (including a weak type $(2,2)$ endpoint), and they stated a condition \cite[Conjecture 1.14]{MR2654385} on $u$ under which they extended \eqref{passumption} to all $p_{0}$ in a neighborhood of $2$.
That condition is known to hold for analytic vector fields, and more generally for a class of vector fields previously considered by Bourgain \cite{MR1009171}.
Lacey and Li have also deduced \eqref{2conclusion} from \eqref{passumption} with $p_{0}=2$ with Lipschitz assumption on the vector field replaced by $C^{1+\eta}$.

The estimate \eqref{passumption} for all $1<p_{0}<\infty$ is known for $1$-parameter vector fields \cite{MR3090145} and vector fields constant along Lipschitz curves \cite{MR3592519}.
In these cases the conclusion \eqref{pconclusion} has been obtained in \cite{MR3148061} and \cite{MR3592519}, respectively.
Our argument for the corollary follows closely \cite{MR3148061}, the main additional observation being that \eqref{passumption} can be used as a black box, whereas in \cite{MR3148061} elements of the proof of this estimate for one-parameter vector fields have been used.

We also recall that the Lipschitz regularity hypothesis in Corollary~\ref{cor:LL-single-band} cannot be substantially relaxed.
Once the segments of integration emanating from the points of a fixed vertical line start to overlap, they may do so in a bad way and one can disprove $L^p$ boundedness by testing on characteristic functions of Perron trees, see e.g.\ \cite[Section X.1]{MR1232192}.

Adding curvature to the picture by defining
\begin{equation}
\label{eq:Stein-directional-op}
H^{(\alpha)}_{u} f (x) := \int_{-1}^{1} f(x+r,y+u(x,y)r^{\alpha}) \frac{\dif r}{r},
\end{equation}
where $r^\alpha$ may be interpreted either as $\abs{r}^\alpha$ or ${\rm sgn}(r)\abs{r}^\alpha$, we may argue similarly as above but remove the conditionality thanks to the results in \cite{arXiv:1610.05203}.
We obtain
\begin{corollary}
\label{cor:Hilbert-curved}
For every $0<\alpha<\infty$, $\alpha\neq 1$, and every $1<p<\infty$, there exits $\epsilon_0>0$ such that for every Lipschitz function $u$ with $\norm{u}_{\Lip}\le \epsilon_0$, we have
\begin{equation}
\norm{H^{(\alpha)}_u f}_p
\lesssim
\norm{f}_p.
\end{equation}
\end{corollary}

Our second main result concerns bounds for the square function of the single scale directional operator
\begin{equation}
\label{eq:single-scale-op}
A_{u,\phi}f(x,y) := \int_{-\infty}^{+\infty} \phi(r) f(x+r,y+u(x,y)r) \dif r
\end{equation}
associated to a Schwartz function $\phi$.
\begin{theorem}
\label{thm:single-scale}
Let $u:\R^{2}\to [-1,1]$ be a measurable function.
Then
\begin{equation}
\label{eq:single-scale-square-fct}
\norm[\big]{\big(\sum_{t \in 2^{\Z}} \abs{ A_{u,\phi} P_{t} f }^{2} \Big)^{1/2}}_{p}
\lesssim_{p,\phi}
\norm{f}_{p}
,
\quad
2<p<\infty.
\end{equation}
\end{theorem}
The operator $A_{u,\phi}$ is in general not bounded on $L^p$ unless $p=\infty$.
Even if we assume $u$ to be Lipschitz in the vertical direction, we may not apply our first main theorem if $\phi$ does not have suitable compact support, and the operator $A_{u,\phi}$ remains unbounded in general.

As an application of this result, we elaborate on a remark made by Demeter in \cite{MR2680067}.
\begin{corollary}\label{cor:single-scale-N-directions}
Assume the measurable function $u:\R^2\to [-1,1]$ takes at most $N$ different values.
Then
\begin{equation} \label{est:single-scale-N-directions}
\norm{A_{u,\phi}f}_p \lesssim_{p,\phi}  \log(N+2)^{1/2} \norm{f}_p,
\quad
2<p<\infty.
\end{equation}
\end{corollary}
Indeed, Demeter proves the sharper endpoint version of this estimate for $p=2$, reproducing an earlier result by Katz \cite{MR1711029}.
Demeter proposes an alternative proof of this result using an inequality by Chang, Wilson, and Wolff \cite{MR800004}, in the same vein as in his proof of \cite[Theorem 2]{MR2680067}.
Theorem~\ref{thm:single-scale} allows to follow through with this proposal, albeit only for $p>2$.
For the operator obtained by replacing $\phi$ in \eqref{eq:single-scale-op} with a one-dimensional singular integral kernel, the same quantitative estimate as \eqref{est:single-scale-N-directions}, up to $\varepsilon$-losses in the power of $\log N$ when $p>2$ is sufficiently close to $2$, holds when the finite range of $u$ is assumed to have additional structure \cite{MR3145928}.
For instance, one may take $u(\R^2)=\{2k/N: k=-N/2,\ldots,N/2\}$.
Thus, it is of interest whether the methods behind Corollary~\ref{cor:single-scale-N-directions} may be applied to the singular integral case, with the aim of lifting the structure restrictions appearing in \cite{MR3145928}.

FDP was partially supported by NSF grants DMS-1500449 and DMS-1650810, by the Severo Ochoa Program SEV-2013-0323 and by Basque Government BERC Program 2014-2017.
SG and CT acknowledge support by the NSF under grant DMS-1440140 through participation in the harmonic analysis program at MSRI in Spring 2017.
CT and PZK acknowledge support by DFG-SFB 1060 and the Hausdorff Center for Mathematics in Bonn.

\section{Lipschitz vector fields}

\subsection{Carleson embeddings with compactly supported test functions}
\label{sec:Carlseon-embeddings-compact-support}
We refer to \cite[Section 2 and 3]{MR3312633} for the general theory of outer measure spaces.
In this section we use the outer measure space $X=\R^{d}\times (0,\infty)$ with the collection of distinguished sets $\bfE$ consisting of the \emph{tents}
\[
T(x,s) = \Set{(y,t) : \norm{x-y}+t \leq s}
\]
and an outer measure $\mu$ generated by $\sigma(T(x,s))=s^{d}$.

Let $\omega$ be a \emph{Dini modulus of continuity}, that is, $\omega : [0,\infty) \to [0,\infty)$ is a function that is subadditive in the sense
\[
u\leq s+t \implies \omega(u) \leq \omega(s) + \omega(t)
\]
and has finite Dini norm $\norm{\omega}_{\Dini} = \int_{0}^{1} \omega(t) \frac{\dif t}{t}$.
Let $\calC$ be the class of testing functions $\phi : \R^{d}\to\C$ that satisfy
\begin{align}
\label{C:cancel}
\textstyle{\int} \phi(z) \dif z &= 0,\\
\label{C:decay}
\supp \phi
&\subset
B(0,1)\\
\label{C:cont}
\abs{\phi(z)-\phi(z')}
&\leq
\omega(\norm{z-z'})
\quad\text{for all } z,z'\in\R^{d}.
\end{align}

For locally integrable functions $f$ we define the embeddings
\begin{align*}
\Aem_{c} f(x,t) &:= t^{-d} \int_{B(x,t)} \abs{f},\\
\Dem_{c} f(x,t) &:= \sup_{\phi\in\calC} \abs[\big]{t^{-d} \int f(z) \phi(t^{-1}(y-z)) dz}.
\end{align*}

\begin{theorem}[{cf.~\cite[Theorem 4.1]{MR3312633}}]
\label{thm:Carlseon-embeddings-compact-support}
For every $1<p\leq\infty$ we have
\begin{align*}
\norm{\Aem_{c} f}_{L^{p}(S^{\infty})} &\lesssim \norm{f}_{L^{p}(\R^{d})},\\
\norm{\Dem_{c} f}_{L^{p}(S^{2})} &\lesssim \norm{f}_{L^{p}(\R^{d})}.
\end{align*}
Moreover, we have the endpoint estimates
\begin{align*}
\norm{\Aem_{c} f}_{L^{1,\infty}(S^{\infty})} &\lesssim \norm{f}_{L^{1}(\R^{d})},\\
\norm{\Dem_{c} f}_{L^{1,\infty}(S^{2})} &\lesssim \norm{f}_{L^{1}(\R^{d})}.
\end{align*}
\end{theorem}
The main difference from \cite[Theorem 4.1]{MR3312633} is the supremum over $\phi\in\calC$ in the definition of $\Dem_{c}$, whereas \cite[Theorem 4.1]{MR3312633} uses a fixed $\phi$.
This supremum does not affect the proof strongly, but is important for our application.
The precise choice of the class of test functions $\calC$ is not important for this application, but the Dini regularity condition appears naturally in the proof.

We linearize the supremum in the definition of $\Dem_{c} f$ by choosing for each pair $(y,t)$ a function $\phi\in\mathcal{C}$ for which the supremum is almost attained.
Denote then $\phi_{y,t}(z) = t^{-d}\phi(t^{-1}(y-z))$.
This is an $L^{1}$ normalized wave packet at scale $t$.
The almost orthogonality of these wave packets is captured by the following estimate.

\begin{lemma}
If $t\leq t'$ then
\label{lem:wave-packet-correlation-decay}
\[
\abs{\innerp{ \phi_{y,t}}{\phi_{y',t'}}}
\lesssim
(t')^{-d} \omega(t/t')
\]
\end{lemma}

\begin{proof}
Using the cancellation condition \eqref{C:cancel} and the support condition we write
\begin{multline*}
\abs[\big]{\int_{\R^{d}} \phi_{y,t} (z) \phi_{y',t'}(z) \dif z}
=
\abs[\big]{\int_{B(y,t)} \phi_{y,t} (z) (\phi_{y',t'}(z) - \phi_{y',t'}(y)) \dif z}\\
\leq
\int_{B(y,t)} \abs{\phi_{y,t} (z)} (t')^{-d} \omega(t/t') \dif z
\lesssim
(t')^{-d} \omega(t/t').
\qedhere
\end{multline*}
\end{proof}

We use the almost orthogonality statement in Lemma~\ref{lem:wave-packet-correlation-decay} to deduce a square function estimate for $p=2$.
\begin{lemma}
\label{lem:L2-square-fct}
\begin{equation}
\label{0114e2.7}
\int_{\R^d \times \R_{>0} } \abs{\Dem_{c} f(y,t)}^2 \dif y \frac{\dif t}{t}
\lesssim
\norm{f}_{2}^{2}.
\end{equation}
\end{lemma}
\begin{proof}
We begin with a measurable selection of functions $\phi_{y,t}$ that almost extremize $\Dem_{c} f(y,t)$.
Expand the square of the left hand side of \eqref{0114e2.7}
\begin{align*}
\left( \int \abs{\innerp{f}{\phi_{y,t}}}^{2} \dif y \frac{\dif t}{t} \right)^2
&=
\left( \int_{\R^d} \left( \int_{\R^d \times \R_{>0}} \innerp{f}{\phi_{y,t}} \phi_{y,t}(z) \dif y \frac{\dif t}{t} \right) \overline{f(z)} \dif z\right)^2 \\
&\leq
\norm[\big]{\int_{\R^d \times \R_{>0}} \innerp{f}{\phi_{y,t}} \phi_{y,t}(z) \dif y \frac{\dif t}{t} }_{2}^{2} \norm{f}_2^2 \\
\intertext{We further expand the square from the former term }
&=
\iint \innerp{f}{\phi_{y,t}} \innerp{\phi_{y,t}}{\phi_{y',t'}} \innerp{\phi_{y',t'}}{f} \dif y \frac{\dif t}{t} \dif y' \frac{\dif t'}{t'} \norm{f}_2^2 \\
&\leq
\int \abs{\innerp{f}{\phi_{y,t}}}^2 \int \abs{\innerp{\phi_{y,t}}{\phi_{y',t'}}} \dif y' \frac{\dif t'}{t'} \dif y \frac{\dif t}{t} \norm{f}_2^2,
\end{align*}
using the estimate
\[
2\abs{\innerp{f}{\phi_{y,t}} \innerp{\phi_{y',t'}}{f}} \leq
\abs{\innerp{f}{\phi_{y,t}}}^2 + \abs{\innerp{f}{\phi_{y',t'}}}^2
\]
in the last inequality.
It suffices to verify
\[
\sup_{y,t} \int \abs{\innerp{\phi_{y,t}}{\phi_{y',t'}}} \dif y' \frac{\dif t'}{t'} < \infty.
\]
By Lemma~\ref{lem:wave-packet-correlation-decay} and using bounded support of the $\phi_{y,t}$'s we have
\begin{align*}
\int \abs{\innerp{\phi_{y,t}}{\phi_{y',t'}}} \dif y' \frac{\dif t'}{t'}
&\lesssim
\int_{t\leq t'} \int_{\norm{y-y'}\leq t+t'} (t')^{-d} \omega(t/t') \dif y' \frac{\dif t'}{t'}\\
\qquad + \int_{t>t'} \int_{\norm{y-y'}\leq t+t'} t^{-d} \omega(t'/t) \dif y' \frac{\dif t'}{t'} \\
&\lesssim
\int_{t\leq t'} \omega(t/t') \frac{\dif t'}{t'}
+ \int_{t>t'} \omega(t'/t) \frac{\dif t'}{t'}
\lesssim
\norm{\omega}_{\Dini}.
\end{align*}
This finishes the proof of Lemma~\ref{lem:L2-square-fct}.
\end{proof}

\begin{proof}[Proof of Theorem~\ref{thm:Carlseon-embeddings-compact-support}]
We may assume that the superlevel sets $\Set{Mf>\lambda}$, where $M$ is the uncentered Hardy--Littlewood maximal function, have finite measure for all $\lambda>0$, since otherwise the right-hand side of the conclusion is infinite.

Let $\Set{Q_{i}}_{i}$ be a Whitney decomposition of the superlevel set $\Set{Mf>\lambda}$.
Let $x_{i}$ denote the center and $r_{i}$ the diameter of $Q_{i}$.
Let
\begin{equation}
E:=\bigcup_{i} T(x_{i},3\sqrt{d}r_{i})
\end{equation}
and note that
\[
\mu(E) \lesssim \meas{\Set{Mf>\lambda}}.
\]
The claim of the theorem will therefore follow from the more precise results
\begin{align}
\label{eq:Aem-outside-whitney-tents}
\norm{\Aem_{c} f \one_{E^{c}}}_{L^{\infty}(S^{\infty})} &\lesssim \lambda,\\
\label{eq:Dem-outside-whitney-tents}
\norm{\Dem_{c} f \one_{E^{c}}}_{L^{\infty}(S^{2})} &\lesssim \lambda.
\end{align}
Let $(x,t)\in E^{c}$.
Then no ball $B(y,t/\sqrt{d})$ with $\norm{x-y}\leq t$ is contained in a Whitney cube.
It follows that, for some constant $C$ that depends only on the dimension, the ball $B(x,Ct)$ is not contained in $\Set{Mf>\lambda}$.
Hence
\[
\Aem_{c} f(x,t)
\leq
t^{-d} \int_{B(x,Ct)} \abs{f}
\lesssim
\lambda.
\]
This completes the proof of \eqref{eq:Aem-outside-whitney-tents}.
Now we show \eqref{eq:Dem-outside-whitney-tents}.
The Calder\'on--Zygmund decomposition $f=g+b$, $b=\sum_{i}b_{i}$ associated to the Whitney decomposition $\Set{Q_{i}}_{i}$ has the properties
\begin{enumerate}
\item $\norm{g}_{\infty} \lesssim \lambda$,
\item $\supp b_{i} \subset Q_{i}$,
\item $\int b_{i} = 0$,
\item $\meas{Q_{i}}^{-1} \int \abs{b_{i}} \lesssim \lambda$.
\end{enumerate}
Using the bounded support condition on the wave packets and Lemma~\ref{lem:L2-square-fct} we obtain
\begin{align*}
S^2(\Dem_{c}g) (T(x,s))
&=
\left( \frac{1}{s^d} \int_{T(x,s)} \abs{\Dem_{c} g(y,t)}^2 \dif y \frac{\dif t}{t}\right)^{1/2}\\
&=
\left( \frac{1}{s^d} \int_{T(x,s)} \abs{\Dem_{c} (g\one_{B(x,2s)})(y,t)}^2 \dif y \frac{\dif t}{t}\right)^{1/2}\\
&\lesssim
s^{-d/2} \norm{g\one_{B(x,2s)}}_{2}
\lesssim
\norm{g}_{\infty}.
\end{align*}
Hence \eqref{eq:Dem-outside-whitney-tents} holds with $f$ replaced by $g$.
By sublinearity of the embedding map $\Dem$ and subadditivity of the outer $L^{\infty}(S^{2})$ norm it remains to show \eqref{eq:Dem-outside-whitney-tents} holds with $f$ replaced by $b$.
More explicitly, for every tent $T=T(x,r)$ we want to show
\[
S^{2}(\Dem_{c} b \one_{E^{c}})(T)
\lesssim
\lambda.
\]
We know
\[
S^{\infty}(\Dem_{c} b \one_{E^{c}})(T)
\lesssim
S^{\infty}(\Dem_{c} f \one_{E^{c}})(T)
+
S^{\infty}(\Dem_{c} g \one_{E^{c}})(T)
\lesssim
S^{\infty}(\Aem_{c} f \one_{E^{c}})(T)
+
\lambda
\lesssim
\lambda.
\]
By logarithmic convexity of $S^{p}$ sizes it therefore suffices to show
\begin{equation}
\label{eq:bad-fct-L1-embed}
S^{1}(\Dem_{c} b \one_{E^{c}})(T)
\lesssim
\lambda.
\end{equation}

\begin{claim}
\label{claim:bad-atom-L1-embed}
$\int_{t>r_{i}} \Dem_{c} b_{i}(x,t) \dif x \frac{\dif t}{t} \lesssim \lambda r_{i}^{d}$.
\end{claim}
\begin{proof}[Proof of Claim~\ref{claim:bad-atom-L1-embed}.]
Notice that, due to support constraints, $\Dem_{c} b_{i}(x,t)$ can only be non-zero if $\norm{x-x_{i}}\leq r_{i}/2+t$.
Moreover, under this condition and choosing $\phi_{x,t}$ that almost extremizes $\Dem_{c} b_{i}(x,t)$ we obtain
\begin{align*}
\abs[\big]{\int b_{i}(z) \phi_{x,t}(z) \dif z}
&=
\abs[\big]{\int_{\norm{z-x_{i}}\leq r_{i}/2} b_{i}(z) (\phi_{x,t}(z) - \phi_{x,t}(x_{i})) \dif z}\\
&\leq
t^{-d} \omega(r_{i}/(2t)) \int_{\norm{z-x_{i}}\leq r_{i}/2} \abs{b_{i}(z)} \dif z\\
&\lesssim
t^{-d} \omega(r_{i}/(2t)) \lambda r_{i}^{d}.
\end{align*}
Hence
\begin{multline*}
\int_{t>r_{i}} \Dem_{c} b_{i}(x,t) \dif x \frac{\dif t}{t}
\leq
\int_{t>r_{i}, \norm{x-x_{i}}\leq t+r_{i}/2} \Dem_{c} b_{i}(x,t) \dif x \frac{\dif t}{t}\\
\lesssim
\int_{t>r_{i}} \omega(r_{i}/(2t)) \lambda r_{i}^{d} \frac{\dif t}{t}
\lesssim
\lambda r_{i}^{d} \norm{\omega}_{\Dini}.
\end{multline*}
This finishes the proof of Claim~\ref{claim:bad-atom-L1-embed}.
\end{proof}

In order to show \eqref{eq:bad-fct-L1-embed} notice that only the Whitney cubes $Q_{i} \subset B(x,10 r)$ contribute to $\Dem_{c} b \one_{T\setminus E}$.
\begin{align*}
S^{1}(\Dem_{c} b \one_{E^{c}})(T)
&=
r^{-d} \int_{T\setminus E} \Dem_{c} b(z,t) \dif z \frac{\dif t}{t}\\
&\leq
r^{-d} \sum_{i : Q_{i} \subset B(x,10 r)} \int_{T\setminus E} \Dem_{c} b_{i}(z,t) \dif z \frac{\dif t}{t}\\
&\leq
r^{-d} \sum_{i : Q_{i} \subset B(x,10 r)} \int_{t>r_{i}} \Dem_{c} b_{i}(z,t) \dif z \frac{\dif t}{t}\\
\intertext{using Claim~\ref{claim:bad-atom-L1-embed}}
&\lesssim
r^{-d} \lambda \sum_{i : Q_{i} \subset B(x,10 r)} \meas{Q_{i}}\\
\intertext{by disjointness of Whitney cubes}
&\lesssim
r^{-d} \lambda \meas{B(x,10 r)}
\lesssim
\lambda.
\end{align*}
This finishes the proof of Theorem~\ref{thm:Carlseon-embeddings-compact-support}.
\end{proof}

\subsection{Carleson embeddings with tails}
It is possible to adapt the proofs in Section~\ref{sec:Carlseon-embeddings-compact-support} to embeddings defined using test functions with tails.
Since we do not need testing functions with sharp decay rates for tails, we will instead estimate such embeddings by averaging the results in Section~\ref{sec:Carlseon-embeddings-compact-support}.

In this section we work in dimension $d=1$ and consider the following embedding maps:
\begin{align}
\Aem f(x,t) &:= \int t^{-1} (1+\abs{x-y}/t)^{-5} \abs{f(y)} \dif y,\\
\Dem f(x,t) &:= \sup_{\phi\in\Phi} \abs[\big]{\int t^{-1} \phi((x-y)/t) f(y) \dif y},
\end{align}
where
\[
\Phi = \Set{ \phi:\R\to\C ,\ \int\phi=0,\ \abs{\phi(x)}\leq (1+\abs{x})^{-10},\ \abs{\phi'(x)}\leq (1+\abs{x})^{-10}}.
\]
The smoothness and decay conditions in these embeddings are not optimal, but they suffice for our purposes.
Decomposing the testing functions $(1+\abs{x})^{-5}$ and $\phi\in\Phi$ into series of compactly supported bump functions as in \cite[Lemma 3.1]{MR2320408}, see also Lemma~\ref{lem:split-wave-packet} in this article, we can deduce the embeddings
\begin{align}
\label{eq:A-embedding}
\norm{\Aem f}_{L^{p}(S^{\infty})} &\lesssim \norm{f}_{p},\\
\label{eq:D-embedding}
\norm{\Dem f}_{L^{p}(S^{2})} &\lesssim \norm{f}_{p}
\end{align}
for $1<p\leq\infty$ from Theorem~\ref{thm:Carlseon-embeddings-compact-support}.

\subsection{Jones beta numbers}
Let $A : \R\to\C$ be a Lipschitz function and let $a$ be its distributional derivative, so that $\norm{a}_{\infty} = \norm{A}_{\Lip}$.
Let $\psi$ be a compactly supported bump function with
\begin{equation}\label{0111e1.6}
\int\psi(x)\dif x = \int x\psi(x) \dif x = 0
\end{equation}
and
\[
\int_{0}^{\infty} \hat\psi(\xi\xi_{0}) \frac{\dif\xi}{\xi} = 1
\quad\text{for}\quad
\xi_{0}\neq 0.
\]
Let $\psi_{t} = t^{-1} \psi(t^{-1}\cdot)$ be an $L^{1}$ normalized mean zero bump function at scale $t$.
Let
\begin{equation}
\label{eq:av-slope}
\alpha(x,t) := \int_{t}^{\infty} a * \psi_{s}(x) \frac{\dif s}{s}
\end{equation}
be the average slope of $A$ near $x$ at scale $t$ and let
\begin{equation}
\label{eq:beta-n-number}
\beta_{n}(x,t) := \sup_{x_{0},x_{1},x_{2} \in B(x,2^{n} \cdot 3t), 2^{-n}t \leq \tilde t \leq 2^{n} t} t^{-1} \abs{A(x_{2})-A(x_{1})-\alpha(x_{0},\tilde t)(x_{2}-x_{1})}.
\end{equation}
This definition includes the supremum over the range of uncertainty around $(x,t)$, which seems convenient.
\begin{lemma}\label{0111lemma1}
With the notation \eqref{eq:beta-n-number} we have
\label{lem:beta-square}
\[
\beta_{n}(x,t)
\lesssim
t^{-1} \int_{0}^{2^{n} t} \Big( \int_{\abs{y-x}\lesssim 2^{n} t} \Dem a(y,s)^{2} s^{-1} \dif y \Big)^{1/2} \dif s
+
2^{2n}\int_{2^{n}t}^{\infty} \Dem a(x,s) \frac{t \dif s}{s^{2}}\\
\]
\end{lemma}
\begin{proof}
Let $x_{0},x_{1},x_{2}\in B(x,2^{n}\cdot 3t)$, $2^{-n}t\leq \tilde t \leq 2^{n}t$.
By the fundamental theorem of calculus and Calder\'on's reproducing formula for $a$ we can write
\begin{align*}
\MoveEqLeft
t^{-1} (A(x_{2})-A(x_{1})-\alpha(x_{0},\tilde t)(x_{2}-x_{1}))
=
t^{-1} \int_{x_{1}}^{x_{2}} a(y) \dif y - t^{-1} \alpha(x_{0},\tilde t)(x_{2}-x_{1})\\
&=
t^{-1} \int_{x_{1}}^{x_{2}} \int_{0}^{\infty} a * \psi_{s}(y) \frac{\dif s}{s} \dif y
-
t^{-1} \int_{x_{1}}^{x_{2}} \int_{\tilde t}^{\infty} a * \psi_{s}(x_{0}) \frac{\dif s}{s} \dif y.
\intertext{Splitting the integral in $s$ in the former term at $\tilde t$ we further obtain }
&=
t^{-1} \int_{x_{1}}^{x_{2}} \int_{0}^{\tilde t} a * \psi_{s}(y) \frac{\dif s}{s} \dif y
+
t^{-1} \int_{x_{1}}^{x_{2}} \int_{\tilde t}^{\infty} (a * \psi_{s}(y) - a * \psi_{s}(x_{0}) ) \frac{\dif s}{s} \dif y\\
&=:
I + II.
\end{align*}
We estimate the two terms on the right-hand side separately.
In the first term we note $\psi_{s} = s (\tilde\psi_{s})'$, where $\tilde\psi_{s}$ is also an $L^{1}$ normalized mean zero bump function at scale $s$, by assumption \eqref{0111e1.6}.
Therefore
\begin{align*}
I
&\leq
t^{-1} \int_{0}^{2^{n} t} \abs[\Big]{ \int_{x_{1}}^{x_{2}}  a * \psi_{s}(y)\ \dif y} \frac{\dif s}{s}\\
&=
t^{-1} \int_{0}^{2^{n} t} \abs{ a * \tilde\psi_{s}(x_{2}) - a * \tilde\psi_{s}(x_{1}) } \dif s\\
&\lesssim
t^{-1} \int_{0}^{2^{n} t} \sup_{\abs{y-x}\lesssim 2^{n}t} \Dem a(y,s) \dif s.
\end{align*}
Since $\Dem a(\cdot,s)$ is almost constant at scale $s$, this can be further estimated by
\[
I
\lesssim
t^{-1} \int_{0}^{2^{n} t} \Big( \int_{\abs{y-x}\lesssim 2^{n}t} \Dem a(y,s)^{2} s^{-1} \dif y \Big)^{1/2} \dif s.
\]
We split the second term $II\leq II_{a}+II_{b}$ via
\begin{equation}\label{0111e1.10}
\int_{\tilde t}^{\infty} \leq \int_{2^{-n}t}^{2^{n}t} + \int_{2^{n}t}^{\infty}.
\end{equation}
Then
\[
II_{a}
\leq
t^{-1} \int_{2^{-n}t}^{2^{n}t} \sup_{\abs{y-x}\lesssim 2^{n} t}\abs{a * \psi_{s}(y)} \frac{\dif s}{s},
\]
and this can be absorbed into the estimate for $I$.
The latter term from \eqref{0111e1.10} is bounded by
\[
II_{b} \lesssim
t^{-1} \int_{x_{1}}^{x_{2}} \int_{2^{n}t}^{\infty} \abs[\big]{a * [\psi_{s}(\cdot-x+y) - \psi_{s}(\cdot-x+x_{0})](x)} \frac{\dif s}{s} \dif y.
\]
Since $\abs{x-y},\abs{x-x_{0}}\lesssim 2^{n} t\lesssim s$, the function in the square brackets is a mean zero $L^{1}$ normalized bump function at scale $s$ with constant $\lesssim 2^{n}t/s$ by the fundamental theorem of calculus, so
\[
II_{b} \lesssim
t^{-1} \int_{x_{1}}^{x_{2}} \int_{2^{n}t}^{\infty} \Dem a(x,s) \frac{2^{n} t \dif s}{s^{2}} \dif y
\lesssim
2^{2n} \int_{2^{n}t}^{\infty} \Dem a(x,s) \frac{t \dif s}{s^{2}}.
\]
This finishes the proof of Lemma~\ref{0111lemma1}.
\end{proof}
\begin{lemma}[{cf.\ \cite[Lemma 3]{MR1013815}}]
$\norm{\beta_{n}}_{L^{\infty}(S^{2})} \lesssim 2^{3n/2} \norm{a}_{\infty}$.
\end{lemma}
\begin{proof}
We have to show
\[
\frac{1}{t_{0}} \int_{t<t_{0}} \int_{\abs{x-x_{0}}<t_{0}} \beta_{n}(x,t)^{2} \dif x \frac{\dif t}{t}
\lesssim
2^{3n} \norm{a}_{\infty}^{2}
\]
with the implicit constant independent of $(x_{0},t_{0})\in\R\times\R_{+}$.

We estimate the $S^{2}$ size on the tent centered at $x_{0}$ with height $t_{0}$ separately for the two terms in the conclusion of Lemma~\ref{lem:beta-square}.
For the first term we consider the square of the $S^{2}$ size:
\begin{align*}
\MoveEqLeft
\frac{1}{t_{0}} \int_{t<t_{0}} \int_{\abs{x-x_{0}}<t_{0}} \Big( t^{-1} \int_{0}^{2^{n} t} \Big( \int_{\abs{y-x}\lesssim 2^{n} t} \Dem a(y,s)^{2} s^{-1} \dif y \Big)^{1/2} \dif s \Big)^{2} \dif x \frac{\dif t}{t}\\
\intertext{Apply H\"older's inequality in the $s$-variable}
&\leq
\frac{1}{t_{0}} \int_{t<t_{0}} \int_{\abs{x-x_{0}}<t_{0}} \int_{0}^{2^{n} t} \Big( \int_{\abs{y-x}\lesssim 2^{n}t} \Dem a(y,s)^{2} \dif y \Big) \frac{\dif s}{s^{1/2}} \cdot \int_{0}^{2^{n} t} \frac{\dif s}{s^{1/2}} \dif x \frac{\dif t}{t^{3}}\\
\intertext{Change the order of integration}
&\lesssim
\frac{2^{n/2}}{t_{0}} \int_{s\leq 2^{n} t_{0}} \int_{2^{-n}s<t<t_{0}} \int_{\abs{y-x_{0}}\lesssim 2^{n}t_{0}} \int_{\abs{x-y}\lesssim 2^{n}t} \dif x \Dem a(y,s)^{2} \dif y \frac{\dif t}{t^{5/2}} \frac{\dif s}{s^{1/2}}\\
&\lesssim
\frac{2^{2n}}{t_{0}} \int_{s\leq 2^{n} t_{0}} \int_{\abs{y-x_{0}}\lesssim 2^{n}t_{0}} \Dem a(y,s)^{2} \dif y \frac{\dif s}{s} \lesssim
2^{3n} \norm{\Dem a}_{L^{\infty}(S^{2})}^{2}.
\end{align*}
For the second term we consider the $S^{2}$ size
\begin{align*}
\MoveEqLeft
2^{2n} \Big(\frac{1}{t_{0}} \int_{t<t_{0}} \int_{\abs{x-x_{0}}<t_{0}} \Big( \int_{2^{n} t}^{\infty} \Dem a(x,s) \frac{t \dif s}{s^{2}} \Big)^{2} \dif x \frac{\dif t}{t} \Big)^{1/2}\\
\intertext{By applying a change of variable $s\to t\tau$ and Minkowski's integral inequality:}
&\leq
2^{2n}\int_{2^{n}}^{\infty} \Big(\frac{1}{t_{0}} \int_{t<t_{0}} \int_{\abs{x-x_{0}}<t_{0}} \Dem a(x,t\tau)^{2} \dif x \frac{\dif t}{t}\Big)^{1/2} \frac{\dif\tau}{\tau^{2}}\\
&=
2^{2n}\int_{2^{n}}^{\infty} \Big(\frac{1}{\tau t_{0}} \int_{s<\tau t_{0}} \int_{\abs{x-x_{0}}<t_{0}} \Dem a(x,s)^{2} \dif x \frac{\dif s}{s}\Big)^{1/2} \frac{\dif\tau}{\tau^{3/2}}\\
&\lesssim
2^{2n}\int_{2^{n}}^{\infty} \norm{\Dem a}_{L^{\infty}(S^{2})} \frac{\dif\tau}{\tau^{3/2}}\lesssim
2^{3n/2} \norm{\Dem a}_{L^{\infty}(S^{2})}.
\end{align*}
The conclusion follows from \eqref{eq:D-embedding}.
\end{proof}

\begin{corollary}[{cf.~\cite[Lemma 4]{MR1013815}}]
\label{cor:jones-beta}
Let $\epsilon>0$ and
\[
\beta(x,t)
=
\sup_{x_{0},x_{1},x_{2} \in \R, \tilde t>0} \big(1 + \frac{\max_{i}(\abs{x_{i}-x})}{t} + \frac{\tilde t}{t} + \frac{t}{\tilde t}\big)^{-3/2-\epsilon} \frac{\abs{A(x_{2})-A(x_{1})-\alpha(x_{0},\tilde t)(x_{2}-x_{1})}}{t}.
\]
Then
\[
\norm{\beta}_{L^{\infty}(S^{2})}
\lesssim
\norm{a}_{\infty}.
\]
\end{corollary}
The difference from the original formulation of Jones's beta number estimate is that we take a supremum over an uncertainty region in all available parameters.

\subsection{Littlewood--Paley diagonalization of Lipschitz change of variables}
\begin{proof}[Proof of Theorem~\ref{thm:Lip-LP-diag}]
Since the Lipschitz norm of $A$ is strictly smaller than $1$, the change of variable $x\mapsto x+A(x)$ is invertible and bi-Lipschitz.
Denote its inverse function by $b$, so that $z = b(z) + A(b(z))$.

Write
\[
T_{A}(\Psi_{t} * f)(x)
=
T_{A}(\Psi_{t} * P_{t}f)(x)
=
\int_{-\infty}^{\infty} P_{t}f(z) \Psi_{t}(x+A(x)-z) \dif z.
\]
This integral is a linear combination of the functions $x\mapsto \Psi_{t}(x+A(x)-z)$ that we view as non-linear deformations of wave packets centered at $b(z)$.
The main idea is to replace the non-linear change of variable $x\mapsto x+A(x)-z$ in the argument of $\Psi_{t}$ by the linear change of variable $x\mapsto (1+\alpha(b(z),t))(x-b(z))$, where $\alpha$ is the average slope of the function $A$ in the sense of \eqref{eq:av-slope}.
Since $\abs{\alpha} \leq \norm{A}_{\Lip}$, the function
\[
x\mapsto \int_{-\infty}^{\infty} P_{t}f(z) \Psi_{t}((1+\alpha(b(z),t))(x-b(z))) \dif z
\]
has Fourier support inside $t^{-1}[99/100,103/100]$, so it is annihilated by $I-P_{t}$.

It remains to estimate the error that has been made in approximating the non-linear change of coordinates in the argument of $\Psi_{t}$ by a linear one.
To this end we compute the difference of the arguments:
\begin{equation}
\label{eq:lin-lip-arg}
\abs{x+A(x)-z - (1+\alpha(b(z),t))(x-b(z))}
=
\abs{(A(x) - A(b(z)) - \alpha(b(z),t)(x-b(z)))}
\end{equation}
By the Lipschitz property of $A$ and since $\abs{\alpha}\leq\norm{A}_{\Lip}$ we have
\[
\eqref{eq:lin-lip-arg}
\leq
\frac12 \abs{x-b(z)},
\]
and it follows that both $x+A(x)-z$ and $(1+\alpha(b(z),t))(x-b(z))$ have (signed) distance of the order $\approx x-b(z)$ from zero.
Therefore
\begin{align*}
\MoveEqLeft
\abs{\Psi_{t}(x+A(x)-z) - \Psi_{t}((1+\alpha(b(z),t))(x-b(z)))}\\
&\lesssim
t^{-2} (1+\abs{x-b(z)}/t)^{-20} \cdot \eqref{eq:lin-lip-arg}
&&\text{by decay of $\Psi_{t}'$}\\
&\lesssim
t^{-1} \beta(b(z),t) (1+ \abs{x-b(z)}/t)^{-10}
&&\text{by definition of $\beta$ numbers}.
\end{align*}
It follows that
\begin{align*}
\MoveEqLeft
\sum_{t \in 2^{\Z}} \abs{(1-P_{t})T_{A} (\Psi_{t} * f)}\\
&=
\sum_{t \in 2^{\Z}} \abs[\big]{(1-P_{t}) \int P_{t}f(z) (\Psi_{t}(x+A(x)-z) - \Psi_{t}((1+\alpha(b(z),t))(x-b(z)))) \dif z}\\
&\lesssim
\sum_{t \in 2^{\Z}} (\delta_{0} + t^{-1}(1+\abs{\cdot}/t)^{-10}) * \int \abs{P_{t}f(z)} t^{-1} \beta(b(z),t) (1+ \abs{x-b(z)}/t)^{-10} \dif z\\
&\lesssim
\sum_{t \in 2^{\Z}} \int \abs{P_{t}f(z)} t^{-1} \beta(b(z),t) (1+ \abs{x-b(z)}/t)^{-5}.
\end{align*}
Multiplying this with a function $g\in L^{p'}(\R)$ and integrating in $x$ we obtain the estimate
\[
\sum_{t\in 2^{\Z}} \int \Dem f(z,t) \beta(b(z),t) \Aem g(b(z),t) \dif z.
\]
The sum over $t$ can be dominated by $\int_{0}^{\infty} \frac{\dif t}{t}$ since all functions $\Dem,\beta,\Aem$ are almost (up to a multiplicative factor) constant on Carleson boxes $B(x,t) \times [t,2t]$.
By \cite[Proposition 3.6]{MR3312633} and outer H\"older inequality \cite[Proposition 3.4]{MR3312633} this is bounded by
\[
\norm{\Dem f}_{L^{p}(S^{2})} \norm{\beta(b(\cdot),\cdot)}_{L^{\infty}(S^{2})} \norm{\Aem g(b(\cdot),\cdot)}_{L^{p'}(S^{\infty})}.
\]
Since the function $b$ is bi-Lipschitz, it does not affect outer norms up to a multiplicative constant.
To see this note that
\[
\norm{F\one_{(\cup_{i} T(x_{i},s_{i}))^{c}}}_{L^{\infty}(S^{q})} \leq \lambda
\implies
\norm{F(b(\cdot),\cdot)\one_{(\cup_{i} T(b^{-1}(x_{i}),2s_{i}))^{c}}}_{L^{\infty}(S^{q})} \leq C\lambda
\]
for a sufficiently large constant $C$.

Thus we obtain the estimate
\[
\norm{\Dem f}_{L^{p}(S^{2})} \norm{\beta}_{L^{\infty}(S^{2})} \norm{\Aem g}_{L^{p'}(S^{\infty})}.
\]
Estimating the first term using \eqref{eq:D-embedding}, the middle term using Corollary~\ref{cor:jones-beta}, and the last term using \eqref{eq:A-embedding} we obtain the claim.
\end{proof}

\subsection{Application to truncated directional Hilbert transforms}
\label{guo-subsection2.5}
\begin{proof}[Proof of Corollary~\ref{cor:LP-diag}]
By Minkowski's integral inequality we obtain
\begin{align*}
\MoveEqLeft
\norm[\Big]{ \sum_{t \in 2^{\Z}} \abs{(1-P_{t})T_{u} (\Psi_{t}*_{2} f)} }_{L^{p}(\R^{2})}\\
&\leq
\int_{r=-1}^{1} \Big(\int_{\R} \norm[\Big]{ \sum_{t \in 2^{\Z}} \abs{(1-P_{t})\big(\Psi_{t} * f(x+r,\cdot+r u(, \cdot))\big)} }_{L^{p}(\R)}^{p} \dif x \Big)^{1/p} \frac{\dif r}{\abs{r}}\\
\intertext{By Theorem~\ref{thm:Lip-LP-diag}, we further obtain}
&\lesssim
\int_{r=-1}^{1} \Big(\int_{\R} \norm{ru(x,\cdot)}_{\Lip}^{p} \norm{ f(x+r,\cdot) }_{L^{p}(\R)}^{p} \dif x \Big)^{1/p} \frac{\dif r}{\abs{r}}\\
&\lesssim
\int_{r=-1}^{1} \Big(\int_{\R} \norm{ f(x+r,\cdot) }_{L^{p}(\R)}^{p} \dif x \Big)^{1/p} \dif r \lesssim
\int_{r=-1}^{1} \norm{ f }_{L^{p}(\R^{2})} \dif r \lesssim
\norm{ f }_{L^{p}(\R^{2})}.
\end{align*}
This finishes the proof of Corollary~\ref{cor:LP-diag}.
\end{proof}

In the remaining part of this section we prove Corollary~\ref{cor:LL-single-band}.
As an initial reduction observe that it suffices to estimate the restriction of $H_{u}$ to a vertical strip; more precisely we need an estimate of the form
\[
\norm{H_{u}f}_{L^{p}([N-1,N+2]\times \R)} \lesssim \norm{H_{u}f}_{L^{p}([N,N+1]\times \R)}
\]
for functions $f$ supported in the vertical strip $[N,N+1]\times\R$.
This reduction will be important in the case $p_{0}<2$.
Also, it is easy to see that we may replace $H_{u}$ by the smoothly truncated operator
\begin{equation}
\label{tildeHu}
\tilde H_u f(x,y):= \operatorname{p.v.}\int_{\R} f(x+r,y+u(x,y)r) \phi(r) \frac{\dif r}r,
\end{equation}
where $\phi$ is a smooth even function with $\phi(0)=1$, $\int\phi(x)\dif x = \int x\phi(x) \dif x = \dotsb = \int x^{N}\phi(x) = 0$ for some large $N$ and $\supp\phi\subset [-1,1]$.
This is possible because the maps $(x,y)\mapsto (x+r,y+u(x,y)r)$ are uniformly bi-Lipschitz for $r\in [-1,1]$, so $f\mapsto f(x+r,y+u(x,y)r)$ is a bounded operator on $L^{p}$.

We note that the operators $f\mapsto H_{u}(\Psi_{t}*f)$ (as well as the analogous ones obtained with $\tilde H_u$ from \eqref{tildeHu} in place of $H_u$) are also trivially bounded in $L^p$ uniformly in $t\geq t_{0}$.
To see this split
\[
H_{u}f(x,y)
=
\int_{-1}^{1} f(x+r,y) \frac{\dif r}{r}
+
\int_{-1}^{1} (f(x+r,y+u(x,y)r)-f(x+r,y)) \frac{\dif r}{r}.
\]
The first term is a one-dimensional truncated Hilbert transform on each horizontal line, and therefore bounded on any $L^{p}$, $1<p<\infty$.
The second term can be written as
\[
\int_{-1}^{1} \int_{s=0}^{u(x,y)r} \partial_{2} f(x+r,y+s) \dif s \frac{\dif r}{r}
\]
This is in turn bounded by
\[
\int_{-1}^{1} M_{2} \partial_{2} f(x+r,y) \dif r\leq
M_{1} M_{2} \partial_{2} f(x,y),
\]
where $M_{i}$ denotes the Hardy--Littlewood maximal function in the $i$-th variable.
The differential operator $\partial_{2}$ is $L^{p}$ bounded on the subspace of functions with $\hat f(\xi,\eta)=0$ for $\abs{\eta}>2/t_{0}$ and therefore we obtain $L^{p}$ estimates for this term.

\begin{remark}
The same argument can be used to estimate $H_{u}$ on functions with small horizontal frequencies, thus simplifying an argument in \cite[Section 3]{arXiv:1603.03317}.
\end{remark}

Below, we work with $\tilde H_u$ from \eqref{tildeHu} in place of $H_u$, and omit the tilde for simplicity of notation. By linearity and the Calder\'on reproducing formula it suffices to estimate the operator
\[
f\mapsto \int_0^\infty H_{u} (\Psi_t *_2 f) \frac{\dif t}t
\]
in $L^{p}$.
By superposition of Corollary~\ref{cor:LP-diag} we obtain $L^{p}$ estimates for the off-diagonal term
\[
f\mapsto \int_0^\infty (1-P_{t}) H_{u} (\Psi_t *_2 f) \frac{\dif t}t,
\]
so it suffices to estimate the diagonal term
\[
f\mapsto \int_0^\infty P_{t} H_{u} (\Psi_t *_2 f) \frac{\dif t}t.
\]
By discretization and Littlewood--Paley theory it suffices to show
\[
\norm{\big(\sum_{t\in 2^{\Z}} \abs{H_{u} (\Psi_t *_2 f)}^{2}\big)^{1/2}}_{p}
\lesssim
\norm{\big(\sum_{t\in 2^{\Z}} \abs{P_t *_2 f}^{2}\big)^{1/2}}_{p},
\]
or, more generally,
\[
\norm{\big(\sum_{t\in 2^{\Z}} \abs{H_{u} (\Psi_t *_2 f_{t})}^{2}\big)^{1/2}}_{L^{p}([N-1,N+2]\times\R)}
\lesssim
\norm{\big(\sum_{t\in 2^{\Z}} \abs{f_{t}}^{2}\big)^{1/2}}_{p}
\]
for arbitrary functions $f_{t}$ supported in the strip $[N,N+1]\times\R$.
In the case $p=p_{0}=2$ this follows immediately from the single band hypothesis \eqref{passumption} and Fubini's theorem.

In order to obtain the larger range of $p$'s in the case $1<p_{0}<2$ we use the technique for proving vector-valued estimates introduced in \cite{MR3148061} (see also \cite{MR3352435} for more applications of this technique).
\begin{theorem}
\label{thm:BT-VV}
Let $1<p,q<\infty$ and let $T_{k} : L^{p,1}(\Omega) \to L^{p,\infty}(\Omega')$ be a sequence of subadditive operators.
Let $0\leq c<1$ and suppose that for every pair of (non-null, finite measure) measurable sets $H\subset\Omega$, $G\subset\Omega'$ with $0<|H|,|G|<\infty$ there exist subsets $H'\subset H$, $G'\subset G$ with
\[
\Big(\frac{|G\setminus G'|}{|G|}\Big)^{1-1/p}
+
\Big(\frac{|H\setminus H'|}{|H|}\Big)^{1/p}
\leq
c
\]
for every $k$ and every function $f$ supported on $H'$ we have
\begin{equation}
\label{eq:BT-VV-Lq}
\|T_{k}f\|_{L^{q}(G')}
\lesssim
(|G|/|H|)^{1/q-1/p} \|f\|_{L^{q}(H')}.
\end{equation}
Then for any functions $f_{k}\in L^{p,1}(\Omega)$ we have
\[
\| (\sum_{k} |T_{k}f_{k}|^{q})^{1/q} \|_{L^{p,\infty}(\Omega')}
\lesssim
\| (\sum_{k} |f_{k}|^{q})^{1/q} \|_{L^{p,1}(\Omega)}.
\]
\end{theorem}
\begin{proof}
By the monotone convergence theorem it suffices to consider a finite sequence of operators as long as we obtain estimates that do not depend on its length.
The hypothesis \eqref{eq:BT-VV-Lq} continues to hold for the operator $T(\vec f) := (\sum_{k} \abs{T_{k}f_{k}}^{q} )^{1/q}$ defined on $\ell^{q}$-valued functions, and we know
\[
\| Tf \|_{L^{p_{0},\infty}(\Omega')}
\lesssim
\| f \|_{L^{p_{0},1}(\Omega,\ell^{q})}
\]
with some constant given by the qualitative boundedness assumption on $T_{k}$'s and depending on the length of the sequence of operators.
By duality of Lorentz spaces this is equivalent to
\[
\int_{G} |Tf|
\leq
B |H|^{1/p} |G|^{1-1/p}
\]
for all finite measure sets $H,G$ and all functions $f:\Omega\to\ell^{q}$ with $|f|\leq\one_{H}$.
We have to find a universal upper bound for $B$.

Let $G,H$ be measurable sets with finite measure and $G',H'$ be the major subsets given by the hypothesis.
Then for any function $f : \Omega\to \ell^{q}$ with $|f| \leq \one_{H'}$ we have
\begin{align*}
\int_{G'} |Tf|
&\leq
\| Tf \|_{L^{q}(G')} \| \one_{G} \|_{L^{q'}}\\
&\lesssim
(|G|/|H|)^{1/q-1/p} \norm{f}_{L^{q}(H',\ell^{q})} |G|^{1/q'}\\
&\lesssim
|H|^{1/p} |G|^{1-1/p}
\end{align*}
by H\"older's inequality and the hypothesis.
It follows that for any function $f : \Omega\to\ell^{q}$ with $|f| \leq \one_{H}$ we have
\begin{align*}
\int_{G} |Tf|
&\leq
C |H|^{1/p} |G|^{1-1/p}
+
\int_{G\setminus G'} |Tf|
+
\int_{G'} |T(f\one_{H\setminus H'})|\\
&\leq
C |H|^{1/p} |G|^{1-1/p}
+
B |H|^{1/p} |G\setminus G'|^{1-1/p}
+
B |H\setminus H'|^{1/p} |G|^{1-1/p}\\
&\leq
(C+cB) |H|^{1/p} |G|^{1-1/p}.
\end{align*}
Taking a supremum over $H,G$ we obtain $B\leq C/(1-c)$.
\end{proof}

Corollary~\ref{cor:LL-single-band} will be obtained via an application of Theorem~\ref{thm:BT-VV}  to the operators $T_{k}f = H_{u}(\Psi_{2^{k}}*_{2}f)$, with the  choice  $q=2$.
The corresponding assumption \eqref{eq:BT-VV-Lq} in Theorem~\ref{thm:BT-VV} will follow by interpolation of the estimates
\begin{equation}
\label{eq:local-restricted}
\int (T_{k} (\one_{H'} \one_{F})) \one_{G'} \one_{E}
\lesssim
|E|^{1/2}|F|^{1/2} (|G|/|H|)^{\alpha} (|E|/|F|)^{\beta},
\end{equation}
where $H\subset [N,N+1]\times\R$, $G\subset [N-1,N+1]\times\R$, $H'\subset H$ and $G'\subset G$ are as in Theorem~\ref{thm:BT-VV}, $F,E\subset\R^{2}$ are arbitrary measurable subsets, $\alpha=1/2-1/p$, and $\beta$ is in a neighborhood of $0$.

\begin{figure}
\begin{tabular}{lp{0.5\textwidth}}
\raisebox{-\totalheight}{\begin{tikzpicture}[
circ/.style={shape=circle, inner sep=1pt, draw},
CFs/.style={shape=rectangle, inner sep=1pt, draw},
LKs/.style={shape=rectangle, inner sep=1pt, draw},
L1s/.style={shape=rectangle, inner sep=1pt, draw},
]
\begin{axis}
[
width=5cm,
height=5cm,
xmin=-0.6,
xmax=0.6,
xtick={-0.4, 0, 0.25, 0.5},
xticklabels={$\frac12-\frac1{p_{0}}$, $0$, $\frac14$, $\frac12$},
ymin=-0.6,
ymax=0.6,
ytick={-0.25, 0, 0.5},
yticklabels={$-\frac14$, $0$, $\frac12$},
,xlabel={$\beta$}
,ylabel={$\alpha$}
]
\draw node (L2) at (axis cs: 0,0) [circ] {}
node (Linf) at (axis cs: 0.5,0) [circ] {}
node (CF) at (axis cs: 0,0.5) [CFs] {}
node (LK) at (axis cs: 0.25,-0.25) [LKs] {}
node (L1) at (axis cs: -0.4,0) [L1s] {};
\draw (L2) to (CF) to (Linf) to (LK) to (L2);
\draw[dashed] (CF) to (L1) to (LK) to (CF);
\end{axis}
\end{tikzpicture}}
&
The estimate \eqref{eq:local-restricted} is known unconditionally in the interior of the solid polygon: the line $\alpha=0$ corresponds to the non-localized estimates in \cite{MR3090145} and the other two endpoints are the localized estimates in \cite{MR3148061}.

In the proof of Corollary~\ref{cor:LL-single-band} we use estimates in the interior of the dashed triangle, whose leftmost vertex is the hypothesis \eqref{passumption}.
\end{tabular}
\caption{Localized estimates for the single band directional Hilbert transform}
\label{fig:loc}
\end{figure}

The set of pairs $(\alpha,\beta)$ for which the estimate \eqref{eq:local-restricted} holds is clearly convex.
Hence it suffices to establish \eqref{eq:local-restricted} near the vertices of the dashed triangle in Figure~\ref{fig:loc}.
The intersection of the line $\beta=0$ with this triangle corresponds to the range of $p$'s claimed in \eqref{pconclusion}.

We will use Estimates 16, 17, 21, and 22 from \cite{MR3148061}, which do not rely on the single parameter assumption on the vector field made in \cite{MR3090145,MR3148061}.
One twist is in the proof of Estimate 21, where we have to use a version of \cite[Theorem 8]{MR3148061} for Lipschitz vector fields.
This result goes back to \cite{MR2219012}; a slightly simplified version of the proof of the required covering lemma in \cite{MR3148061} is presented in Appendix~\ref{sec:LL}.
The covering lemma for Lipschitz vector fields only holds for parallelograms of bounded length.
This is the reason for restricting the operator $H_{u}$ to a vertical strip: we can apply the covering lemma to the intersection of parallelograms with this vertical strip.
The other difficulty is that we are dealing with a (smooth) truncation of the Hilbert kernel, so the results of \cite{MR3090145} do not directly apply.
The easiest way to work around this seems to be running the argument in \cite{MR3090145} with more general wave packets which can be used to assemble also the truncated Hilbert kernel $\phi(r)/r$.

\subsubsection{Using the single band estimate below $L^{2}$}
The hypothesis \eqref{passumption} shows in particular that \eqref{eq:local-restricted} holds with $(\alpha,\beta)=(0,1/2-1/p_{0})$.

\subsubsection{Using the C\'ordoba--Fefferman covering argument}
By Estimates 16, 17, and 22 in \cite{MR3148061} we can estimate the left-hand side of \eqref{eq:local-restricted} by
\[
\sum_{\delta} \sum_{\sigma \lesssim \delta^{-n} (|G|/|H|)^{n-1}} \min(|F|\delta\sigma^{-1},|E|\sigma)
\]
for any integer $n\geq 2$, where both sums are over positive dyadic numbers.

The (geometric) sum over $\sigma$ has two critical points: $\sigma \sim \delta^{-n} (|G|/|H|)^{n-1}$ and $\sigma \sim (\delta |F|/|E|)^{1/2}$.
This  gives the estimate
\[
\sum_{\delta} \min((\delta |F| |E|)^{1/2},|E|\delta^{-n} (|G|/|H|)^{n-1}).
\]
The sum over $\delta$ has a critical point with $\delta_{0}^{2n+1} \sim (|G|/|H|)^{2n-2} (|E|/|F|)$, and we obtain the estimate
\[
(\delta_{0} |F| |E|)^{1/2}
\sim
(|F| |E|)^{1/2} (|G|/|H|)^{(n-1)/(2n+1)} (|E|/|F|)^{1/(4n+2)}.
\]
This proves the claim with $\alpha=(n-1)/(2n+1)$, $\beta=1/(4n+2)$.
We can make $(\alpha,\beta)$ approach $(1/2,0)$ by choosing $n$ suitably large.

\subsubsection{Using the Lacey--Li covering argument}
By Estimates 16, 17, and 21 from \cite{MR3148061} we can estimate the left-hand side of \eqref{eq:local-restricted} by
\[
\sum_{\delta} \sum_{\sigma} \min(|F|\delta\sigma^{-1},|E|\sigma,|E| (|H|/|G|)^{1/2}\sigma^{-\epsilon} \delta^{-1/2-\epsilon})
\]
The sum over $\sigma$ now has two critical points with $\sigma\sim (\delta |F|/|E|)^{1/2}$ and with $\sigma^{1+\epsilon}\sim (|H|/|G|)^{1/2} \delta^{-1/2-\epsilon}$ and is dominated by the minimum of the two corresponding terms, so we have the estimate
\[
\sum_{\delta\leq 1} \min(|E| ((|H|/|G|)^{1/2} \delta^{-1/2-\epsilon})^{1/(1+\epsilon)}, (\delta |F| |E|)^{1/2})
\]
The sum over $\delta$ has a critical point at $\delta_{0}^{2+3\epsilon} \sim (|E|/|F|)^{1+\epsilon} (|H|/|G|)$.
This gives the estimate
\[
(\delta_{0} |F| |E|)^{1/2}
\sim
(|F| |E|)^{1/2} ((|E|/|F|)^{1+\epsilon} (|H|/|G|))^{1/(4+6\epsilon)}.
\]
Making $\epsilon$ small we can make $(\alpha,\beta)$ approach $(-1/4,1/4)$.
This completes the proof of Corollary~\ref{cor:LL-single-band}.

\begin{remark}
The upper part of the solid polygon in Figure~\ref{fig:loc} yields the hypothesis of Theorem~\ref{thm:BT-VV} for any $2<q<p<\infty$.
This implies that the operator $H_{u}$ maps $L^{p}(\R^{2})$ into a directional Triebel--Lizorkin space of type $F^{0}_{p,q}$ (provided that $u$ is Lipschitz in the vertical direction).
More precisely,
\[
\norm{\big(\sum_{t\in 2^{\Z}} \abs{P_{t} H_{u} f}^{q} \big)^{1/q} }_{L^{p}(\R^{2})}
\lesssim
\norm{f}_{L^{p}(\R^{2})},
\quad
2<p,q<\infty.
\]
Indeed, the left-hand side is monotonically decreasing in $q$, so it suffices to consider $2<q<p<\infty$.
With a suitable choice of $\Psi$ we may write $f=\sum_{t\in 2^{\Z/100}}\Psi_{t} *_{2} f$.
For notational simplicity we consider only the contribution of $t\in 2^{\Z}$.
By the Fefferman--Stein maximal inequality we may replace $P_{t}$ by larger Littlewood--Paley projections such that $\sum_{t\in 2^{\Z}} P_{t} = \operatorname{id}$.

In the diagonal term we use the Fefferman--Stein maximal inequality, the vector-valued estimate provided by Theorem~\ref{thm:BT-VV} with $p>2$, monotonicity of $\ell^{q}$ norms, and Littlewood--Paley theory to estimate
\begin{align*}
\norm{\big(\sum_{t\in 2^{\Z}} \abs{P_{t} H_{u} (\Psi_{t} *_{2} f)}^{q} \big)^{1/q} }_{L^{p}(\R^{2})}
&\lesssim
\norm{\big(\sum_{t\in 2^{\Z}} \abs{H_{u} (\Psi_{t} *_{2} f)}^{q} \big)^{1/q}}_{L^{p}(\R^{2})}\\
&\lesssim
\norm{\big(\sum_{t\in 2^{\Z}} \abs{\Psi_{t} *_{2} f}^{q} \big)^{1/q}}_{L^{p}(\R^{2})}\\
&\leq
\norm{\big(\sum_{t\in 2^{\Z}} \abs{\Psi_{t} *_{2} f}^{2} \big)^{1/2}}_{L^{p}(\R^{2})}\\
&\lesssim
\norm{f}_{L^{p}(\R^{2})}.
\end{align*}
In the off-diagonal term we use monotonicity of $\ell^{q}$ norms, Littlewood--Paley theory, and Corollary~\ref{cor:LP-diag} to estimate
\begin{align*}
\norm{\big(\sum_{t\in 2^{\Z}} \abs{P_{t} H_{u} (\sum_{t'\neq t}\Psi_{t'} *_{2} f)}^{q} \big)^{1/q} }_{L^{p}(\R^{2})}
&\leq
\norm{\big(\sum_{t\in 2^{\Z}} \abs{P_{t} H_{u} (\sum_{t'\neq t}\Psi_{t'} *_{2} f)}^{2} \big)^{1/2} }_{L^{p}(\R^{2})}\\
&\lesssim
\norm{\sum_{t'\in 2^{\Z}} (\sum_{t\neq t'}P_{t}) H_{u} (\Psi_{t'} *_{2} f) }_{L^{p}(\R^{2})}\\
&=
\norm{\sum_{t'\in 2^{\Z}} (1-P_{t'}) H_{u} (\Psi_{t'} *_{2} f) }_{L^{p}(\R^{2})}\\
&\lesssim
\norm{f }_{L^{p}(\R^{2})}.
\end{align*}
\end{remark}

\subsection{Application to Hilbert transforms along Lipschitz variable parabolas}
\begin{proof}[Proof of Corollary~\ref{cor:Hilbert-curved}]

In the following, we will assume for notational convenience that $0<u\le 1$ almost everywhere.
The region that $-1\le u< 0$ can be handled similarly, while the region $u=0$ is trivial by Fubini as the operator acts only in the first variable.
By the trivial analogue of Corollary~\ref{cor:LP-diag}, it suffices to show
\begin{equation}
\label{main-square}
\norm[\big]{\big(\sum_{t\in2^\Z} \abs{H_{u}^{(\alpha)} (\Psi_t *_2  f)}^2 \big)^{1/2} }_p
\lesssim
\norm{ f }_p.
\end{equation}
We use $P_t(\Psi_t*_2 f)=\Psi_t*_2 f$ where $P_t$ is as defined before acting in the second variable.
We note  that for 
\[\abs{r}^\alpha u(x,y)/t \le 1\]
we have by an application of the fundamental theorem of calculus
\[\abs{P_{t}(\Psi_t *_2 f)(x+r, y+u(x, y)r^\alpha)-P_{t}(\Psi_t *_2 f)(x+r, y)}\le 
u(x,y)\abs{r}^\alpha t^{-1} M_2(\Psi_t *_2 f)(x+r, y).\]
Hence we have for the integral over small values of $r$ 
\[
\norm[\big]{\big(\sum_{t\in2^\Z} \abs{\int_{\abs{r}^\alpha u(x,y)/t\leq 1} 
P_{t}(\Psi_t *_2 f)(x+r, y+u(x, y)r^{\alpha})\frac{\dif r}{r}}^2 \big)^{1/2} }_{L^{p}(x,y)}
\]
\begin{equation}\label{thiele11}
\lesssim \norm[\big]{\big(
\sum_{t\in2^\Z} \abs{\int_{\abs{r}^\alpha u(x,y)/t \leq 1} 
P_t(\Psi_t *_2 f)
(x+r, y) \frac{\dif r}{r}}^2 
\big)^{1/2} }_{L^{p}(x,y)}
\end{equation}
\begin{equation}\label{thiele12}
+
\norm[\big]{\big(
\sum_{t\in2^\Z} \abs{\int_{\abs{r}^\alpha u(x,y)/t \leq 1} 
u(x,y)\abs{r}^\alpha t^{-1} M_2(\Psi_t *_2 f
)(x+r, y)\frac{\dif r}{\abs{r}}}^2 
\big)^{1/2} }_{L^{p}(x,y)}.
\end{equation}
The former term \eqref{thiele11} can be estimated using the vector-valued estimate for the maximally truncated Hilbert transform.
Using integrability of $\abs{r}^{\alpha-1}$ near zero we estimate the latter term \eqref{thiele12} by
\[
\norm[\big]{\big(\sum_{t\in2^\Z} \abs{M_{1} M_{2} 
(\Psi_t *_2 f) 
(x, y)}^2 \big)^{1/2} }_{L^{p}(x,y)}\lesssim
\norm[\big]{\big(\sum_{t\in2^\Z} \abs{ \Psi_t *_2 f(x, y)}^2 \big)^{1/2} }_{L^{p}(x,y)} \lesssim 
\norm{ f }_p.
\]
Here we have used the Fefferman--Stein maximal inequality and Littlewood-Paley theory.

We turn to the remaining part of the kernel with $\abs{r}^\alpha  u(x,y)/t \ge 1$ and $\abs{r}\le 1$.
Note we may restrict the summation over $t$ to $t\le 1$, as for $t>1$ the domain of
integration is empty.
We will break up the integral into lacunary pieces parametrized
by $s\in 2^{\alpha \N}$ and estimate the pieces separately, with suitable power decay 
in $s$ allowing to geometrically sum the estimates.

We introduce Littlewood-Paley projections in the first variable  and write $P_t^{(1)}$
and $P_t^{(2)}$ to distinguish projections in first and second variable.
Consider the averaging operator
\[
E_s^{(1)}=\int_{s}^\infty P_t^{(1)} \frac {dt}t.
\]
We note similarly to above for the averaged part of the integral pieces:
\[
\norm[\big]{\big(\sum_{t\in2^{-\N}} \abs{\int_{s\le \abs{r}^\alpha u(x,y)/t \leq 2^\alpha  s} 
E_{s(\frac{st}{u(x, y)})^{1/\alpha}}^{(1)} P_{t}^{(2)}(\Psi_t *_2 f)(x+r, y+u(x, y)r^{\alpha})\frac{\dif r}{r} }^2 \big)^{1/2} }_{L^{p}(x,y)}
\]
\begin{equation}\label{thiele21}
\lesssim
\norm[\big]{\big(\sum_{t\in2^{-\N}} \abs{\int_{s\le \abs{r}^\alpha u(x,y)/t \leq 2^\alpha s} 
E_{s(\frac{st}{u(x, y)})^{1/\alpha}}^{(1)} P_{t}^{(2)}(\Psi_t *_2 f)(x, y+u(x, y)r^{\alpha})\frac{\dif r}{r} }^2 \big)^{1/2} }_{L^{p}(x,y)}
\end{equation}
\begin{equation}\label{thiele22}
+
\norm[\big]{\big(\sum_{t\in2^{-\N}} (\int_{s\le \abs {r}^\alpha u(x,y)/t\leq 2^{\alpha}s} 
s^{-1} M_1 P_{t}^{(2)}(\Psi_t *_2 f)(x, y+u(x, y)r^{\alpha})\frac{\dif r}{\abs{r}} )^2 \big)^{1/2} }_{L^{p}(x,y)}
\end{equation}
The factor $(st/u)^{1/\alpha}$ in the index of the averaging operator is chosen because it is
roughly $\abs{r}$ in the domain of integration.
In the former term \eqref{thiele21} we change variables, replacing $u(x,y)r^\alpha$ by $r$ on the positive and similarly on the negative axis and do a partial integration in $r$, noting that by the mean zero property the primitive of the kernel of $P_t^{(t)}$ is a bump function again, to estimate this term  by 
\[\lesssim
\norm[\big]{\big(\sum_{t\in2^{-\N}} (\int_{s\leq \abs{r}/t\leq 2^\alpha s} 
t M_1  M_2 (\Psi_t *_2 f)(x, y+r)\frac{\dif r}{\abs{r}^2} )^2 \big)^{1/2} }_{L^{p}(x,y)}
\]
\[\lesssim
s^{-1} \norm[\big]{\big(\sum_{t\in 2^{-\N}}   
(M_2 M_1  M_2 (\Psi_t *_2 f)(x, y) )^2 \big)^{1/2} }_{L^{p}(x,y)}
\]
plus two similar boundary terms, which are all estimated by the Fefferman-Stein maximal inequality
with power decay in $s$.
The latter term \eqref{thiele22} above is estimated by the same change of variables by
\[s^{-1}
\norm[\big]{\big(\sum_{t\in2^{-\N}} (\int_{s\le \abs {r}/t\leq 2^\alpha s} 
 M_1 P_{t}^{(2)}(\Psi_t *_2 f)(x, y+r)\frac{\dif r}{\abs{r}} )^2 \big)^{1/2} }_{L^{p}(x,y)}
\]
\[\lesssim s^{-1}
\norm[\big]{\big(\sum_{t\in2^{-\N}} (
 M_2 M_1 P_{t}^{(2)}(\Psi_t *_2 f)(x, y) )^2 \big)^{1/2} }_{L^{p}(x,y)}
\]
which is again estimated by the Fefferman-Stein maximal inequality with decay in $s$.

A similar estimate can be obtained if instead of the sharp cut-off 
$s\le \abs {r}^\alpha u(x,y)/t\leq 2^{\alpha} s$ we use a smooth cut-off.
More precisely, we will
choose cut-off functions as defined in the following operator:

\begin{equation}
A_s f(x, y)
=
\int_{\R} f(x+r, y+u(x,y) r^\alpha) 
\chi(s^{-1} r^\alpha v(x,y)t^{-1} (u(x,y)v^{-1}(x,y))^{\alpha/(\alpha-1)}) \frac{dr}{r},
\end{equation}
where $\chi$ is smooth and supported on $\pm [2^{-\alpha} ,2^\alpha]$
and $\sum_{s\in 2^{\alpha \N}} \chi(s^{-1}x)=1$ for $x\neq 0$, and where
$v(x,y)$ is the largest integer power of $2$ less than $u(x,y)$.
Note the
auxiliary factor $u/v$ is bounded above and below respectively by $2$ and $1$.

Then, with the above arguments, it suffices to estimate the rough part of each piece
with some $\gamma>0$ that may depend on $p$ as follows:
\begin{equation}\label{guo-e2.40}
\norm[\big]{\big(\sum_{t\in2^{-\N}} \abs{
A_s (1-E_{s(\frac{st}{u(x, y)})^{1/\alpha}}^{(1)}) P_{t}^{(2)}(\Psi_t *_2 f) }^2 \big)^{1/2} }_{L^{p}}\lesssim s^{-\gamma} \norm{f}_p.
\end{equation}

Here we point out that this estimate has essentially been established in \cite{arXiv:1610.05203}.
First of all, we recognize that the left hand side of \eqref{guo-e2.40} is essentially the term (5.13) in \cite{arXiv:1610.05203}, there one has a large power of $s$ in the index of $E$ but this makes their bound
only stronger.
By the local smoothing estimates and a certain interpolation argument, the $L^p$ bounds of \eqref{guo-e2.40} for all $1<p\le 2$ have been established in Subsection 5.3 in \cite{arXiv:1610.05203}.
To prove $L^p$ bounds for all $p>2$, we cite the pointwise estimate (3.19) in \cite{arXiv:1610.05203}, which implies for these $p$ that 

\[
\norm[\big]{\big(\sum_{t\in2^{-\N}} \abs{
A_s (1-E_{s(\frac{st}{u(x, y)})^{1/\alpha}}^{(1)}) P_{t}^{(2)}(\Psi_t *_2 f) }^2 \big)^{1/2} }_{L^{p}}\lesssim \log(1+s)^4 \norm{f}_p.
\]

A further interpolation gives the desired estimate \eqref{guo-e2.40} for all $1<p<\infty$
for slightly smaller $\gamma$.
This finishes the proof of the square function estimate \eqref{main-square}.
\end{proof}

\section{Single scale operator}
\label{sec:single-scale}
In this section we prove Theorem~\ref{thm:single-scale}.
The strategy is to use duality and outer H\"older inequality to reduce the estimate to two estimates of Carleson embedding flavor, the ``energy embedding'' in Section~\ref{sec:single-scale:energy-embed} and the ``mass embedding'' in Section~\ref{sec:single-scale:mass-embed}.

\subsection{Tiles and the outer measure space}
We subdivide the parameter space into \emph{tiles}.
Each tile can be represented in three equivalent ways:
\begin{enumerate}
\item by a \emph{shearing matrix}
\[
A =
\begin{pmatrix}
2^{k_{1}} & 0\\
l2^{k_{1}} & 2^{k_{2}}
\end{pmatrix},
\quad
k_{1},k_{2},l\in\Z
\]
and the spatial location $(2^{-k_{1}}n_{1},2^{-k_{2}}n_{2})$, $n_{1},n_{2}\in\Z$.
\item by the corresponding \emph{spatial parallelogram}
\[
P = A^{-1}([0,1]\times [0,1]) + (2^{-k_{1}}n_{1},2^{-k_{2}}n_{2}),
\]
\item or by the corresponding \emph{frequency parallelogram} $A^{*}([0,1]\times [1,2])$ and the spatial location
\[
(2^{-k_{1}}n_{1},2^{-k_{2}}n_{2}).
\]
\end{enumerate}
Figure~\ref{fig:supp-phiA} shows the spatial and the frequency parallelograms of a tile (with $n_{1}=n_{2}=0$).
The frequency picture also includes the symmetric parallelograms $A^{*}([0,1]\times [-2,-1])$ (in a lighter shade of gray), because the Fourier transforms of the wave packets associated to tiles will concentrate on both these parallelograms.
However, for combinatorial purposes it suffices to consider only the upper parallelogram.
The \emph{slope} of a tile is the number $-l2^{-k_{2}+k_{1}}$.
It is the slope of the lower and the upper side of the corresponding spatial parallelogram.
The spatial parallelogram seems to be the most concise description of a tile, so we denote tiles by the letter $P$ (for ``parallelogram'').

The fact that we are dealing with a single scale operator in Section~\ref{sec:single-scale} is reflected in that we define an outer measure on a finite set $X$ of tiles with $k_{1}=0$, that is, tiles with the fixed horizontal scale $1$.
(The restriction to finite sets of tiles avoids technicalities associated with infinite sums.
All estimates will be independent of the specific finite set, so we can pass to the set of all tiles at the end of the argument.)
The outer measure is generated by a function $\sigma$ whose domain $\mathbf{E}$ is the collection of all non-empty subsets of $X$.
We denote by $CP$ the parallelogram with the same slope and center as $P$ but side lengths multiplied by $C$.
For $\calR\in\mathbf{E}$ set
\begin{equation}
\label{eq:def-sigma}
\sigma(\calR)
:=
\sup_{L\geq 1} L^{-C} \meas[\big]{\cup_{R\in\calR} LR},
\end{equation}
where $C$ is a large number to be chosen later.
The three sizes that we need are
\begin{align*}
S^{1}(F)(\calR)
&:=
\sigma(\calR)^{-1} \sum_{R\in\calR} \meas{R} \abs{F(R)},\\
S^{2}(F)(\calR)
&:=
\big( \sigma(\calR)^{-1} \sum_{R\in\calR} \meas{R} \abs{F(R)}^{2} \big)^{1/2}
=
S^{1}(F^{2})(\calR)^{1/2},\\
S^{\infty}(G)(\calR)
&:=
\sup_{R\in\calR} \abs{G(R)}.
\end{align*}

\begin{figure}
\begin{center}
\begin{tabular}{ccc}
\begin{tikzpicture}
\filldraw[gray] (0,0) -- (0,1) -- (1,1) -- (1,0) -- cycle;
\draw[->] (0,0) -- (1.5,0) node[below] {$x_{1}$};
\draw[->] (0,0) -- (0,1.5) node[left] {$x_{2}$};
\draw (0,0) node[below] {$0$};
\draw (1,0) node[below] {$1$};
\draw (0,1) node[left] {$1$};
\end{tikzpicture}
&
$\xrightarrow{A^{-1} = \begin{pmatrix}
2^{-k_{1}} & 0\\
-l2^{-k_{2}} & 2^{-k_{2}}
\end{pmatrix}}$
&
\begin{tikzpicture}[yscale=0.5]
\filldraw[gray] (0,0) -- (0,1) -- (1,3) -- (1,2) -- cycle;
\draw[->] (0,0) -- (1.5,0) node[below] {$x_{1}$};
\draw[->] (0,0) -- (0,3.5) node[left] {$x_{2}$};
\draw (0,0) node[below] {$0$};
\draw (1,0) node[below] {$2^{-k_{1}}$};
\draw (0,1) node[left] {$2^{-k_{2}}$};
\draw[dotted] (1,2) -- (0,2) node[left] {$-l2^{-k_{2}}$};
\draw[dotted] (1,3) -- (0,3) node[left] {$(-l+1)2^{-k_{2}}$};
\end{tikzpicture}\\
\begin{tikzpicture}
\filldraw[gray] (0,1) -- (0,2) -- (1,2) -- (1,1) -- cycle;
\filldraw[lightgray] (0,-1) -- (0,-2) -- (1,-2) -- (1,-1) -- cycle;
\draw[->] (0,0) -- (1.5,0) node[below] {$\xi_{1}$};
\draw[->] (0,-2.1) -- (0,2.5) node[left] {$\xi_{2}$};
\draw (0,0) node[left] {$0$};
\draw (1,0) node[below] {$1$};
\draw (0,1) node[left] {$1$};
\draw (0,2) node[left] {$2$};
\end{tikzpicture}
&
$\xrightarrow{A^{*} = \begin{pmatrix}
2^{k_{1}} & l2^{k_{1}}\\
0 & 2^{k_{2}}
\end{pmatrix}}$
&
\begin{tikzpicture}[xscale=0.6]
\filldraw[gray] (-2,1) -- (-4,2) -- (-3,2) -- (-1,1) -- cycle;
\filldraw[lightgray] (2,-1) -- (4,-2) -- (5,-2) -- (3,-1) -- cycle;
\draw[->] (-4.5,0) -- (5.5,0) node[below] {$\xi_{1}$};
\draw[->] (0,0) -- (0,2.5) node[left] {$\xi_{2}$};
\draw (0,0) node[below] {$0$};
\draw[dotted] (-2,1) -- (-2,0) node[below] {\scriptsize{$l2^{k_{1}}$}};
\draw[dotted] (-1,1) -- (-1,0) node[below] {\scriptsize{$(l+1)2^{k_{1}}$}};
\draw (0,1) node[left] {$2^{k_{2}}$};
\draw (0,2) node[left] {$2^{k_{2}+1}$};
\end{tikzpicture}
\end{tabular}
\end{center}
\caption{Spatial and frequency parallelograms of a tile}
\label{fig:supp-phiA}
\end{figure}
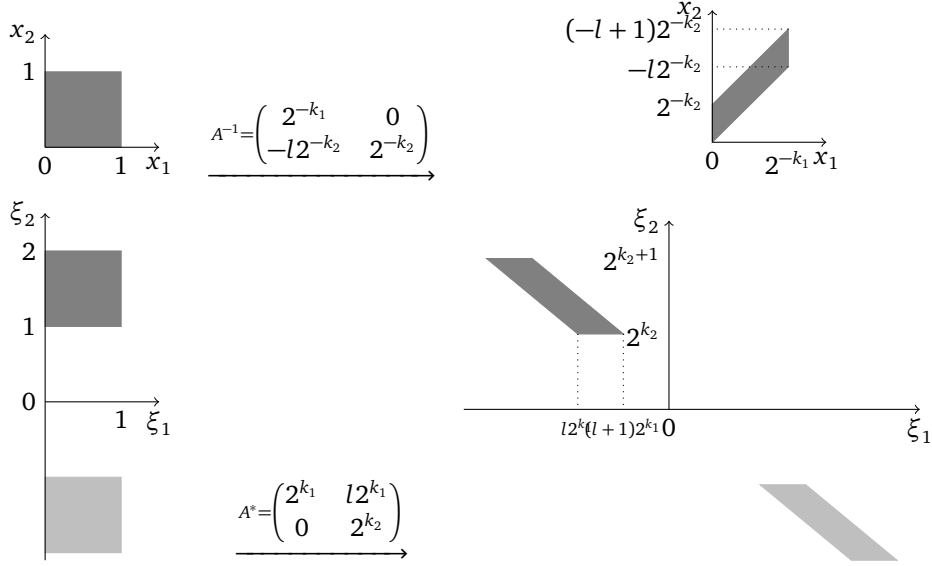

\subsection{Wave packets and the energy embedding}
\label{sec:single-scale:energy-embed}
Let $\Phi=\Phi_{C}$ be the set of functions on $\R^{2}$ that satisfy
\[
\abs{\partial^{\alpha}\phi(x)} \leq (1+\abs{x})^{-C},
\quad
\norm{\alpha}_{\ell^{1}} \leq C,
\]
for some sufficiently large $C$ that will be chosen later and
\[
\int_{\R} x_{2}^{n} \phi(x_{1},x_{2}) \dif x_{2} = 0,
\quad
x_{1}\in\R,
\quad
n=0,\dotsc,C-2.
\]
We think of $\phi$ as morally supported on $[0,1]^{2}$ and of $\hat\phi$ as morally supported on $[0,1]\times [1,2]$ for $\phi\in\Phi$.

The \emph{$L^{\infty}$ normalized wave packets} associated to a tile $P=(A,n_{1},n_{2})$ are the functions of the form
\[
\phi_{P}^{(\infty)}(x) = \phi(A(x_{1}-2^{-k_{1}}n_{1},x_{2}-2^{-k_{2}}n_{2})),
\quad
\phi\in\Phi.
\]
The \emph{$L^{p}$ normalized wave packets}, $1\leq p<\infty$, are the functions $\phi_{P}^{(p)} = \det(A)^{1/p} \phi_{P}^{(\infty)}$.
Note that $\widehat{\phi(A\cdot)}(\xi) = (\det A)^{-1} \hat\phi(A^{-*}\xi)$.
The spatial and the frequency parallelograms of a tile correspond to the moral space/frequency support of the wave packets associated to this tile.

\subsubsection{Almost orthogonality}
The fundamental property of the wave packets is their almost orthogonality for tiles with different scales or slopes.
\begin{lemma}
\label{lem:correlation-decay}
\[
\abs{\innerp{\phi_{P}^{(2)}}{\phi_{P'}^{(2)}}}
\lesssim
\min(1,(2^{\max(k_{2},k_{2}')}\abs{2^{-k_{2}}l-2^{-k_{2}'}l'})^{-C},2^{-C\abs{k_{2}-k_{2}'}}),
\]
where $C$ can be made arbitrarily large provided that the order of decay in the definition of $\Phi$ is sufficiently large.
\end{lemma}
\begin{proof}
Without loss of generality suppose $k_{2}\geq k_{2}'$.
We will estimate
\[
\int_{\R^{2}} \abs{\widehat{\phi(A\cdot)}} \abs{\widehat{\phi'(A'\cdot)}}
\]
for $\phi,\phi'\in\Phi$.
This is sufficient because the spatial location of the tiles only affects the phase of the Fourier transforms of the associated wave packets, but not their magnitude.\\

\paragraph{Correlation decay due to shearing}
Let $0<\epsilon\ll 1$ and $S_{N}=\Set{-N,N}\times\R$ be a vertical strip of width $N \geq 1$.
The critical intersection $A^{*}S_{N} \cap (A')^{*}S_{N}$ is a parallelogram centered at zero of width $\sim N$ and height $\sim N/\abs{2^{-k_{2}}l-2^{-k_{2}'}l'}$.
By the vanishing moments assumption we have
\[
\abs{\widehat{\phi(A\cdot)}} \lesssim 2^{-k_{2}} (2^{-k_{2}}N/\abs{2^{-k_{2}}l-2^{-k_{2}'}l'})^{C}
\]
on the critical intersection.
Using the fact that the Fourier transforms $\widehat{\phi(A\cdot)}$ and $\widehat{\phi'(A'\cdot)}$ are $L^{1}$ normalized functions and the decay of these Fourier transforms at infinity we obtain
\begin{align*}
\int_{\R^{2}} \abs{\widehat{\phi(A\cdot)}} \abs{\widehat{\phi'(A'\cdot)}}
&\leq
\int_{\R^{2} \setminus A^{*}S_{N}} + \int_{\R^{2} \setminus (A')^{*}S_{N}} + \int_{A^{*}S_{N} \cap (A')^{*}S_{N}}\\
&\leq
\sup_{\R^{2} \setminus A^{*}S_{N}} \abs{\widehat{\phi(A\cdot)}}
+
\sup_{\R^{2} \setminus (A')^{*}S_{N}} \abs{\widehat{\phi'(A'\cdot)}}
+
\sup_{A^{*}S_{N} \cap (A')^{*}S_{N}} \abs{\widehat{\phi(A\cdot)}}\\
&\lesssim
2^{-k_{2}}N^{-C(1/\epsilon-1)} + 2^{-k_{2}'}N^{-C(1/\epsilon-1)} + 2^{-k_{2}} (2^{-k_{2}}N/\abs{2^{-k_{2}}l-2^{-k_{2}'}l'})^{C}.
\end{align*}
Choosing $N = 2^{\epsilon (k_{2}-k_{2}')/C} (2^{k_{2}}\abs{2^{-k_{2}}l-2^{-k_{2}'}l'})^{\epsilon}$ as we may provided that $\abs{l-2^{k_{2}-k_{2}'}l'} \geq 1$, we obtain
\[
\int_{\R^{2}} \abs{\widehat{\phi(A\cdot)}} \abs{\widehat{\phi'(A'\cdot)}}
\lesssim
2^{-k_{2}+\epsilon (k_{2}-k_{2}')} (2^{k_{2}}\abs{2^{-k_{2}}l-2^{-k_{2}'}l'})^{-(1-\epsilon)C},
\]
and this gives the second estimate in the conclusion of the lemma.\\

\paragraph{Correlation decay for separated scales}
Let $2^{k_{2}'} \ll N \ll 2^{k_{2}}$.
Using again the fact that the Fourier transforms $\widehat{\phi(A\cdot)}$ and $\widehat{\phi'(A'\cdot)}$ are $L^{1}$ normalized functions and the decay of Fourier transforms near $\xi_{2}=0$ and at infinity we obtain
\begin{align*}
\int_{\R^{2}} \abs{\widehat{\phi(A\cdot)}} \abs{\widehat{\phi'(A'\cdot)}}
&\leq
\int_{\abs{\xi_{2}} \leq N} + \int_{\abs{\xi_{2}} \geq N}\\
&\leq
\sup_{\abs{\xi_{2}} \leq N} \abs{\widehat{\phi(A\cdot)}}
+
\sup_{\abs{\xi_{2}} \geq N} \abs{\widehat{\phi'(A'\cdot)}}\\
&\lesssim
2^{-k_{2}} (N/2^{k_{2}})^{C} + 2^{-k_{2}'}(N/2^{k_{2}'})^{-(C+1)/\epsilon}.
\end{align*}
Choosing $N \sim 2^{k_{2}' + \epsilon(k_{2}-k_{2}')}$ we obtain
\[
\int_{\R^{2}} \abs{\widehat{\phi(A\cdot)}} \abs{\widehat{\phi'(A'\cdot)}}
\lesssim
2^{-k_{2}-C(1-\epsilon)(k_{2}-k_{2}')}
=
2^{-k_{2}/2-k_{2}'/2-(C+1/2-\epsilon')(k_{2}-k_{2}')},
\]
and this gives the third estimate in the conclusion of the lemma.
\end{proof}

\subsubsection{Bessel inequality}
\begin{lemma}
\label{lem:bessel}
For each tile $P$ fix an $L^{2}$ normalized wave packet $\phi_{P}$ adapted to $P$.
Then
\[
\sum_{P} \abs{\innerp{f}{\phi_{P}}}^{2}
\lesssim
\norm{f}_{2}^{2}.
\]
\end{lemma}
\begin{proof}
Schur's test
\begin{align*}
\sum_{P} \abs{\innerp{f}{\phi_{P}}}^{2}
&=
\innerp[\big]{f}{\sum_{P} \phi_{P}\innerp{\phi_{P}}{f}}\\
&\leq
\norm{f}_{2} \norm[\big]{\sum_{P} \phi_{P}\innerp{\phi_{P}}{f}}_{2}\\
&=
\norm{f}_{2} \big(\sum_{P,P'} \innerp{f}{\phi_{P}}\innerp{\phi_{P}}{\phi_{P'}}\innerp{\phi_{P'}}{f} \big)^{1/2}\\
&\leq
\norm{f}_{2} \big(\sum_{P} \abs{\innerp{f}{\phi_{P}}}^{2} \sum_{P'} \abs{\innerp{\phi_{P}}{\phi_{P'}}} \big)^{1/2}
\end{align*}
shows that it suffices to prove
\[
\sup_{P} \sum_{P'} \abs{\innerp{\phi_{P}}{\phi_{P'}}} < \infty.
\]
For a fixed tile $P$ we split the above sum according to the shearing matrix $A'$ of the tile $P'$.
For a given shearing matrix $A'$ we distinguish the cases $k_{2}\leq k_{2}'$ and $k_{2}>k_{2}'$.

In the case $k_{2}\leq k_{2}'$ the tile $P$ has larger scale than $P'$, so the tail of the associated wave packet is more important.
For $L\in 2^{\N}$ let
\[
\tilde\calR_{L} := \Set{ P' \text{ with shearing matrix }A' \text{ such that } LP\cap P'\neq \emptyset }
\]
and let $\calR_{1}:=\tilde\calR_{1}$, $\calR_{L}:=\tilde\calR_{L} \setminus \tilde\calR_{L/2}$ for $L\geq 2$.
Then
\[
\abs{\tilde\calR_{L}}
\lesssim
L(L2^{-k_{2}}+\abs{2^{-k_{2}}l-2^{-k_{2}'}l'})/2^{-k_{2}'},
\]
and
\[
\sum_{L\in 2^{\N}} \sum_{P' \in \calR_{L}} \abs{\innerp{\phi_{P}}{\phi_{P'}}}
\lesssim
\sum_{L\in 2^{\N}} \abs{\tilde\calR_{L}} \min( L^{-C}, 2^{-C(k_{2}'-k_{2})}, (2^{k_{2}'}\abs{2^{-k_{2}}l-2^{-k_{2}'}l'})^{-C})),
\]
where the first estimate inside the minimum is due to spatial separation and the other two estimates come from Lemma~\ref{lem:correlation-decay}.
Summing this over $k_{2}'\geq k_{2}$ and $l'$ we obtain
\begin{multline*}
\sum_{L\in 2^{\N}, k_{2}'\geq k_{2}, l'\in\Z} L(L2^{-k_{2}}+\abs{2^{-k_{2}}l-2^{-k_{2}'}l'})/2^{-k_{2}'} \min( L^{-C}, 2^{-C(k_{2}'-k_{2})}, (2^{k_{2}'}\abs{2^{-k_{2}}l-2^{-k_{2}'}l'})^{-C})\\
\lesssim
\sum_{L\in 2^{\N}, k\geq 0, l'\in\Z} L(L2^{k}+\abs{2^{k}l-l'}) \min( L^{-C}, 2^{-Ck}, \abs{2^{k}l-l'}^{-C})\\
\lesssim
\sum_{L\in 2^{\N}, k\geq 0, l'\in\Z} L(L2^{k}+\abs{2^{k}l-l'}) ( L + 2^{k} + \abs{2^{k}l-l'} )^{-C}
\leq
C.
\end{multline*}

In the region $k_{2}\geq k_{2}'$ we make a similar decomposition with
\[
\tilde\calR_{L} := \Set{ P' \text{ with shearing matrix }A' \text{ such that } P\cap LP'\neq \emptyset }.
\]
The resulting estimate is similar to the above with the roles of $k_{2}$ and $k_{2}'$ reversed.
\end{proof}

\subsubsection{Splitting into compactly supported wave packets}
In order to obtain a localized Bessel inequality we decompose wave packets into compactly supported parts as in \cite[Lemma 3.1]{MR2320408}.
\begin{lemma}
\label{lem:split-wave-packet}
For every $C$ there exists $C'$ such that if $C'\phi \in \Phi_{C'}$, then there exists a decomposition
\[
\phi = \sum_{k\geq 0} 2^{-Ck} \phi_{k},
\quad
\phi_{k}\in\Phi_{C},
\supp\phi_{k} \subset B(0,2^{k}).
\]
\end{lemma}
\begin{proof}[Sketch of proof]
Let $\psi$ be a smooth function supported on $B(0,1/2)$ and identically equal to $1$ on $B(0,1/4)$.
Write $\psi_{k}(x)=\psi(2^{-k}x)$ for its $L^{\infty}$ dilates.
Let also $\eta^{(0)},\dotsc,\eta^{(C-2)}$ be smooth functions supported on $[-1/2,1/2]$ with
\[
\int x^{n} \eta^{(m)}(x) \dif x = \one_{n=m}.
\]
For $k\in\N$ and $x_{1}\in\R$ let
\[
m_{k}^{(n)}(x_{1}) := \int_{\R} x_{2}^{n} \phi(x_{1},x_{2})\psi_{k}(x_{1},x_{2}) \dif x_{2},
\]
then for $\abs{\alpha}\leq C$ and $n<C$ we have
\[
\abs{\partial^{\alpha} m_{k}^{(n)}(x_{1})}
=
\abs[\big]{\int_{\R} x_{2}^{n} \partial_{1}^{\alpha}\phi(x_{1},x_{2})(\psi_{k}(x_{1},x_{2})-1) \dif x_{2}}
\lesssim
2^{-Ck}(1+\abs{x_{1}})^{-C}
\]
provided that $C'$ is sufficiently large.
The claimed splitting is given by
\[
\phi_{k} :=
\begin{cases}
\phi(\psi_{k}-\psi_{k-1}) - \sum_{n=0}^{C-2}(m_{k}^{(n)}-m_{k-1}^{(n)}) \otimes \eta^{(n)}, & k>0,\\
\phi\psi_{0} - \sum_{n=0}^{C-2} m_{0}^{(n)} \otimes \eta^{(n)}, & k=0.
\end{cases}
\qedhere
\]
\end{proof}

\subsubsection{Energy embedding}
The energy embedding is defined by
\[
F(R) := \sup_{\phi_{R}^{(1)}} \abs{\innerp{f}{\phi_{R}^{(1)}}},
\quad
R\in X,
\]
where the supremum is taken over all $L^{1}$ normalized wave packets adapted to $R$ with a sufficiently large order of decay $C'$.

\begin{lemma}
\label{lem:energy-embed-L2}
$\norm{F}_{L^{2,\infty}(S^{2})} \lesssim \norm{f}_{2}$.
\end{lemma}
\begin{proof}
Let $\calR$ be a maximal collection of tiles with $S^{2}(F)(\calR) \geq \lambda$.
If $\calR'\subset X\setminus\calR$ also has size $\geq\lambda$, then using subadditivity of $\sigma$ it is easy to see that $\calR\cup\calR'$ also has size $\geq\lambda$, contradicting maximality.
Hence by maximality we have $\operatorname{outsup}_{X\setminus\calR} S^{2}(F)\leq\lambda$.
On the other hand,
\[
\sigma(\calR)
\leq
\lambda^{-2} \sum_{R\in\calR} \meas{R} \abs{F(R)}^{2}
\lesssim
\lambda^{-2} \norm{f}_{2}^{2}
\]
by Lemma~\ref{lem:bessel}.
\end{proof}

\begin{lemma}
\label{lem:energy-embed-Linfty}
$\norm{F}_{L^{\infty}(S^{2})} \lesssim \norm{f}_{\infty}$.
\end{lemma}
\begin{proof}
Let $\calR\in\bfE$ and let $\phi_{R}$, $R\in\calR$, be wave packets that almost extremize $F(R)$.
Splitting the corresponding members of $\Phi_{C'}$ using Lemma~\ref{lem:split-wave-packet} we obtain decompositions $\phi_{R} = \sum_{k\geq 0}2^{-Ck}\phi_{R,k}$, where each $\phi_{R,k}$ is an $L^{1}$ normalized wave packet adapted to $R$ (with a lower order of decay $C$) and supported on $2^{k}R$.

By Lemma~\ref{lem:bessel} and the support condition we have
\begin{align*}
\sum_{R\in\calR} \meas{R} \abs{\innerp{f}{\phi_{R,k}}}^{2}
&\lesssim
2^{-2Ck} \norm[\big]{f \one_{\cup\Set{2^{k}R : R\in\calR}}}_{2}^{2}\\
&\leq
2^{-2Ck} \norm{f}_{\infty}^{2} \meas[\big]{ \cup_{R\in\calR} 2^{k}R}\\
&\leq
2^{(C_{\ref{eq:def-sigma}}-2C)k} \norm{f}_{\infty}^{2} \sigma(R),
\end{align*}
and summing in $k$ we obtain
\[
\sum_{R\in\calR} \meas{R} \abs{\innerp{f}{\phi_{R}}}^{2}
\lesssim
\norm{f}_{\infty}^{2} \sigma(R),
\]
so that $S^{2}(F)(\calR) \lesssim \norm{f}_{\infty}$ as required.
\end{proof}

\subsection{Covering lemma for parallelograms and the mass embedding}
\label{sec:single-scale:mass-embed}
For completeness we include a slightly streamlined proof of a covering lemma from \cite{MR3148061}.
Covering  lemmas of this type go back to \cite{MR0379785}.
We consider parallelograms with two vertical edges as shown below:
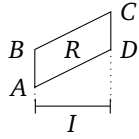
\begin{center}
\begin{tikzpicture}[yscale=0.5]
\draw (0.5,1) node {$R$};
\draw (0,0) node[left] {$A$} -- (0,1) node[left] {$B$} -- (1,2) node[right] {$C$} -- (1,1) node[right] {$D$} -- cycle;
\draw[dotted] (0,0) -- (0,-0.5);
\draw[dotted] (1,1) -- (1,-0.5);
\draw[|-|] (0,-0.5) -- (1,-0.5) node[midway,below] {$I$};
\end{tikzpicture}
\end{center}
The \emph{height} $H(R)$ is the common length of $AB$ and $CD$.
The \emph{shadow} $I(R)$ is the projection of $R$ onto the horizontal axis.
The \emph{slope} $s(R)$ is the common slope of the edges $BC$ and $AD$.
The \emph{uncertainty interval} $U(R) \subset\R$ is the interval between the slopes of $BD$ and $AC$.
It is the interval of length $2H(R)/\meas{I(R)}$ centered at $s(R)$.

\begin{lemma}[{cf.\ \cite[Lemma 7]{MR3148061}}]
\label{lem:bt-cf-covering}
Let $\calR$ a finite collection of parallelograms with vertical edges and dyadic shadow.
Then there exists $\calG\subset\calR$ such that
\begin{equation}
\label{bad}
\meas{ \bigcup _{R \in \calR} R }
\lesssim
\sum_{R\in\calG } \meas{R}
\end{equation}
and for every $n\in\N$ we have
\begin{equation}
\label{eq:good-U}
\sum _{R_{1},\dotsc,R_{n}\in\calG : U(R_{1})\cap\dotsb\cap U(R_{n}) \neq\emptyset } \meas{R_{1}\cap\dotsb\cap R_{n}}
\lesssim_{n}
\sum_{R\in\calG } \meas{R}.
\end{equation}
In particular, for every measurable function $u:\R^{2}\to \R$ the sets
\[
E(R):=\Set{(x,y)\in R: u(x,y)\in U(R)}.
\]
satisfy
\begin{equation}
\label{good}
\int ( \sum _{R\in\calG } \one _{E(R)} )^{q}
\lesssim_{q}
\sum_{R\in\calG } \meas{R},
\quad 0<q<\infty.
\end{equation}
\end{lemma}

In \cite{MR3148061} the conclusion \eqref{good} is stated for one-variable vector fields, but this structural assumption is not used in the proof.

In the proof of Lemma~\ref{lem:bt-cf-covering} we denote by $CR$ the parallelogram with the same center, slope, and shadow as $R$ but height $CH(R)$ (this definition of $CR$ is used only here).
We need the following geometric observation:
\begin{lemma}\label{7rlemma}
Let $R,R'$ be two parallelograms with $I(R)= I(R')$, $U(R)\cap U(R')\neq\emptyset$, and $R\cap R'\neq \emptyset$.
If  $7H(R)\le H(R')$, then $7R\subseteq 7R'$.
\end{lemma}

Let $M_V$ denote the Hardy--Littlewood maximal operator in the vertical direction:
\begin{equation}
\label{eq:MV}
M_Vf(x,y)
=
\sup_{y\in J} \meas{J}^{-1} \int_J \abs{f(x,z)} \dif z,
\end{equation}
where the supremum is taken over all intervals $J$ containing $y$.

\begin{proof}[Proof of Lemma~\ref{lem:bt-cf-covering}]
We select $\calG$ using the following iterative procedure.
Initialize
\begin{align*}
STOCK &:= \calR \\
\calG &:= \emptyset.
\end{align*}
While $STOCK \neq \emptyset$, choose an $R\in STOCK$ with maximal $\meas{I(R)}$.
Update
\begin{align*}
\calG &:= \calG \cup \Set{R},\\
STOCK &:= STOCK \setminus \Set{R\in STOCK : R\subset \Set{M_{V} ( \sum_{R'\in\calG} \one _{7R'}) \geq 10^{-4}}}.
\end{align*}
This procedure terminates after finitely many steps since at each step at least the selected parallelogram $R$ is removed from $STOCK$.

By construction
\begin{align}
\bigcup _{R\in\calR} R
\subset
\Set{ x\colon M_{V} ( \sum_{R\in\calG} \one_{7R}) (x) \geq 10^{-4}},
\end{align}
and \eqref{bad} follows by the weak $(1,1)$ inequality for $M_{V}$.

We prove \eqref{eq:good-U} by induction on $n$.
For $n=1$ the statement clearly holds.
Suppose that \eqref{eq:good-U} holds for a given $n$, we will show that it also holds with $n$ replaced by $n+1$.
For each $R'\in\calG$ let
\[
\calG(R') := \Set{R\in\calG \text{ chosen prior to } R' \text{ with } R\cap R'\neq\emptyset, U(R)\cap U(100 R')\neq \emptyset}.
\]
All terms in \eqref{eq:good-U} in which some $R_{i}$ occurs at least twice are estimated by the inductive hypothesis.
In the remaining terms we may arrange the $R_{i}$'s in the order reverse to the selection order (losing a factor $(n+1)!$), and omitting some vanishing terms we obtain the estimate
\begin{equation}
\label{eq:adm-7r}
\sum_{R_{0}\in\calG, R_{1}\in\calG(R_{0}), \dotsc, R_{n} \in\calG(R_{n-1})}
\meas{R_{0} \cap \dotsb \cap R_{n}}
\leq
\sum_{R_{0}\in\calG, R_{1}\in\calG(R_{0}), \dotsc, R_{n} \in\calG(R_{n-1})}
\meas{I(R_{0})} \cdot \meas{H(R_{n})}.
\end{equation}
We claim that for every $R' \in \calG$ we have
\begin{equation}
\label{eq:adm2}
\sum_{R\in \calG(R')} H(R) \leq H(R').
\end{equation}
To see this let $R\in\calG(R')$, so that in particular $I(R')\subset I(R)$ and $U(R)\cap U(10R') \neq\emptyset$.
If $H(R')\leq H(R)$, then $7H(10R')\leq H(70R)$, and Lemma~\ref{7rlemma} shows that $70R'\subset 490 R$, so that $M_{V}(\one_{R})\geq 490^{-1}$ on $R'$, contradicting $R'\in\calG$.
Therefore $H(R') > H(R)$, so $7H(R) \leq H(10R')$, and Lemma~\ref{7rlemma} shows that
\[
7R \cap (I(R') \times\R) \subset 70 R'.
\]
The  inequality \eqref{eq:adm2} follows, since otherwise $M_{V}(\sum_{R\in\calG(R')}\one_{R}) \geq 70^{-1}$ on $R'$, contradicting $R'\in\calG$.
Hence
\begin{align*}
\eqref{eq:adm-7r}
&\leq
\sum_{R_{0}\in\calG, R_{1}\in\calG(R_{0}), \dotsc, R_{n-1} \in\calG(R_{n-2})}
\meas{I(R_{0})} \cdot \meas{H(R_{n-1})}\\
&\leq \dotsb \leq
\sum_{R_{0}\in\calG}
\meas{I(R_{0})} \cdot \meas{H(R_{0})}
=
\sum_{R_{0}\in\calG} \meas{R_{0}}.
\end{align*}
This completes the proof of \eqref{eq:good-U}.
In order to see \eqref{good} observe that its left-hand side is monotonically increasing in $q$, so it suffices to consider integer values $q=n$, and in this case the left-hand side of \eqref{good} is dominated by the left-hand side of \eqref{eq:good-U}.
\end{proof}

\subsubsection{Mass embedding}
The mass embedding is given by
\[
G(R) := \meas{R}^{-1} \int_{E_{R}} \abs{g},
\quad
R\in X.
\]
\begin{lemma}
\label{lem:mass-embedding}
Let $1<q<\infty$.
If the constant $C$ in the definition of $\sigma$ is sufficiently large depending on $q$, then $\norm{G}_{L^{q,\infty}(S^{\infty})} \lesssim \norm{g}_{q}$.
\end{lemma}
Recall that $CP$ now again denotes the parallelogram $P$ expanded by the factor $C$ both in the horizontal and in the vertical direction.
\begin{proof}
Let $\delta>0$, $g\in L^{q}(\R^{2})$, and let $\calR$ be a collection of tiles such that $G(R)\geq\delta$ for $R\in\calR$ .
We have to show
\begin{equation}
\label{eq:bt-ct:est}
\sup_{L\geq 1} L^{-C}\meas[\big]{\cup_{R\in\calR} LR}
\lesssim_{q}
\delta^{-q} \norm{g}_{q}^{q}.
\end{equation}
Note that the definition of $G(R)$ makes sense for arbitrary parallelograms (not only the dyadic ones that we call tiles).
For the enlarged parallelograms $LR$ we still have $G(LR) \geq \delta/L^{2}$, so it suffices to show \eqref{eq:bt-ct:est} with $L=1$ and a collection of arbitrary parallelograms $\calR$, provided that the constant $C_{\ref{eq:def-sigma}}$ in the definition of $\sigma$ is at least $2q$.

Enlarging the parallelograms in such a way that their shadows become intervals in adjacent dyadic grids and the uncertainty intervals stay the same we preserve the hypothesis $G(R)\gtrsim\delta$ up to a multiplicative constant.
Hence we may assume that the parallelograms have dyadic shadows.

In view of \eqref{bad} it suffices to consider the parallelograms in the subset $\calG\subset\calR$ provided by Lemma~\ref{lem:bt-cf-covering}.
By the density assumption and H\"older's inequality we have
\begin{align*}
\sum_{R\in\calG} \meas{R}
&\leq
\sum_{R\in\calG} \frac{1}{\delta} \int_{E(R)} \abs{g}\\
&=
\frac{1}{\delta}
\norm[\big]{ \sum_{R\in\calG} \one_{E(R)} \abs{g} }_1\\
&\leq
\frac{1}{\delta} \norm{\sum_{R\in\calG} \one_{E(R)}}_{q'} \norm{g}_{q}\\
&\lesssim
\frac{1}{\delta} \left(\sum_{R\in\calG} \meas{R} \right)^{1/q'} \norm{g}_{q},
\end{align*}
where in the last passage we have used the estimate \eqref{good}.
After division by the middle factor of the right hand side we obtain the claim.
\end{proof}

\subsection{Estimate for the square function}
We finally prove Theorem~\ref{thm:single-scale}.
Note that $A_{u,\phi} P_{2,t} f(x)$ is the integral of $f$ against an $L^{1}$ normalized wave packet associated to a tile that contains $x$ and whose uncertainty interval contains $u(x)$.
Hence the left-hand side of \eqref{eq:single-scale-square-fct} is bounded by
\[
\norm{ \big( \sum_{R\in X} F(R)^{2} \one_{E_{R}} \big)^{1/2} }_{p}
=
\norm{ \sum_{R\in X} F(R)^{2} \one_{E_{R}} }_{p/2}^{1/2}.
\]
Dualizing with a function $g\in L^{(p/2)'}$ we obtain
\[
\int \sum_{R\in X} F(R)^{2} \one_{E_{R}} g
=
\sum_{R\in X} \meas{R} F(R)^{2} G(R).
\]
For every $\calR\in\bfE$ we have $\sum_{R\in\calR} \meas{R} F(R) = \sigma(\calR) S^{1}(F)(\calR)$.
Therefore by \cite[Proposition 3.6]{MR3312633} and outer H\"older inequality \cite[Proposition 3.4]{MR3312633} the above is bounded by
\[
\norm{ F^{2} G }_{L^{1}(S^{1})}
\lesssim
\norm{ F^{2} }_{L^{p/2}(S^{1})} \norm{ G }_{L^{(p/2)'}(S^{\infty})}
=
\norm{ F }_{L^{p}(S^{2})}^{2} \norm{ G }_{L^{(p/2)'}(S^{\infty})}.
\]
The first term is bounded by $\norm{f}_{p}^{2}$ by Lemmas \ref{lem:energy-embed-L2} and \ref{lem:energy-embed-Linfty} and interpolation \cite[Proposition 3.5]{MR3312633}.
The second term is bounded by $\norm{g}_{(p/2)'}$ by Lemma~\ref{lem:mass-embedding} and interpolation \cite[Proposition 3.5]{MR3312633}.

\subsection{Application to a maximal operator with a restricted set of directions}
\label{sec:single-scale-N-directions}

In this section we prove Corollary~\ref{cor:single-scale-N-directions}.

Although the operator \eqref{eq:single-scale-op} is unbounded for general direction fields $u$, it is clearly bounded (on any $L^{p}$, $1\leq p\leq\infty$) with norm $O(N)$ as long as $u$ is allowed to take at most $N$ values.
This trivial estimate has been improved to $O(\sqrt{\log N})$ on $L^{2}$ by Katz \cite{MR1711029}.
Note that we also have the trivial estimate $O(1)$ on $L^{\infty}$, and by interpolation one obtains logarithmic dependence on $N$ of the operator norm of \eqref{eq:single-scale-op} on $L^{p}$ also for all $2<p<\infty$.
Demeter \cite{MR2680067} gives an alternative proof of Katz's result, and furthermore hints at yet another different proof   via reduction to the square function bound Theorem~\ref{thm:single-scale} by means of the good-$\lambda$ inequality with sharp constant due to Chang, Wilson, and Wolff \cite{MR800004}.
The first appearance of a similar reduction to   square function   in the context of maximal multipliers   goes back to Grafakos, Honz\'ik, and Seeger \cite{MR2249617}, and analogous approaches have been since used in Demeter \cite{MR2680067} and Demeter with the first author \cite{MR3145928}.
We have not been able to reproduce the endpoint $p=2$ using this technique.
However, notice that our   square function approach,  after interpolation, recovers the result for $p>2$ up to an arbitrarily small loss in the exponent of the logarithm.

\begin{proof}[Proof of Corollary \ref{cor:single-scale-N-directions}]
For $j\in \Z$, define the dyadic martingale averaging operator
\begin{equation}
E_j f :=\sum 2^{2j}  \innerp{f}{\one_{Q}} \one_{Q} ,
\end{equation}
where the summation runs over all standard dyadic squares $Q$ in $\R^2$ with side length $2^{-j}$.
Further define
\[
\Delta_j =E_{j+1}-E_j ,
\]
\[
\Delta f:=(\sum_{j\in \Z}\abs{\Delta_j f}^2)^{1/2}.
\]
Let $M$ denote the non-dyadic Hardy--Littlewood maximal operator.
Chang, Wilson, and Wolff \cite[Corollary 3.1]{MR800004} prove that there are universal constants $c_1$ and $c_2$ such that for all $\lambda>0$ and $0<\epsilon<1$
\begin{equation}
\label{eq:cww}
\meas{ \Set{z: \abs{f(z)-E_0 f(z)}>2\lambda,\  \Delta f(z)\le \epsilon \lambda} }
\le
c_2 e^{-\frac{c_1}{\epsilon^2}}  \meas{ \Set{z: Mf(z)\ge \lambda} }.
\end{equation}

Denote the finitely many values of $u$ by  $u_i$, $1\le i\le N$, and write $A_{u_i}$ for the operator with the constant direction field $u_i$.
Corollary~\ref{cor:single-scale-N-directions} follows by Marcinkiewicz interpolation from the weak type inequality
\[
\meas{\Set{z:\sup_i \abs{A_{u_i}f(z)}>4 \lambda}}
\le
C \log(N+2)^{p/2} \lambda^{-p}\norm{f}_p^p
\]
for $2<p<\infty$.
Gearing up for Chang, Wilson, and Wolff we estimate
\begin{align}
\nonumber\MoveEqLeft
\meas{\Set{z:\sup_i \abs{A_{u_i}f(z)}>4 \lambda}}\\
&=
\nonumber
\meas{\bigcup_i \Set{z:\abs{A_{u_i}f(z)}>4 \lambda}}\\
\label{firstcww}
&\le
\meas{\bigcup_i \Set{z:\abs{A_{u_i}f(z)-E_0A_{u_i}f(z)}>2 \lambda, \Delta A_{u_i}f(z)\le \epsilon \lambda }}\\
\label{secondcww}
&\quad+
\meas{\bigcup_i \Set{z:\abs{E_0A_{u_i}f(z)}>2 \lambda}}\\
\label{thirdcww}
&\quad+
\meas{\bigcup_i \Set{z:\Delta A_{u_i} f(z)> \epsilon\lambda  }}
\end{align}
Using \eqref{eq:cww} we estimate
\begin{align*}
\eqref{firstcww}
&\le
\sum_i \meas{\Set{z:\abs{A_{u_i}f(z)-E_0A_{u_i}f(z)}>2 \lambda, \Delta A_{u_i}f(z)\le \epsilon \lambda }}\\
&\le
C \sum_{i} e^{-\frac{c_1}{\epsilon^2}}\meas{\Set{z: M(A_{u_i}f)(z)>2 \lambda}}\\
&\le C \sum_{i} e^{-\frac{c_1}{\epsilon^2}} \lambda^{-p}\norm{A_{u_i}f}_p^p\\
&\le C N e^{-\frac{c_1}{\epsilon^2}} \lambda^{-p}\norm{f}_p^p\\
&\le C\lambda^{-p}\norm{f}_p^p
\end{align*}
provided $\epsilon \leq c_{1}^{1/2}\log (N+2)^{1/2}$.

The function $E_0 A_{u_i}f$ in \eqref{secondcww} is pointwise dominated by the standard Hardy--Littlewood maximal operator, because $E_0$ and $A_ {u_i}$ compose to some averaging operator at scale $0$.
Therefore
\[
\eqref{secondcww}
\leq
\meas{\Set{z: Mf(z)>C\lambda}}
\lesssim
\lambda^{-p}\norm{f}_p^p.
\]

To control \eqref{thirdcww} we introduce a suitable Littlewood--Paley decomposition in the second variable, note that $P_{2^k}$ commutes with $A_{u_i}$, and estimate pointwise
\begin{align*}
\sup_i \Delta A_{u_i}f
&=
\sup_i
\Delta (\sum_{k\in \Z} P_{2^k}A_{u_i}P_{2^k} f) \\
&=
\sup_i
(\sum_j \abs{\sum_{k}  \Delta_j P_{2^k}A_{u_i}P_{2^k} f}^2)^{1/2} \\
&\lesssim
\sup_i
(\sum_j (\sum_{k}  2^{-\abs{j-k}/q'} M M_{q,V} A_{u_i}P_{2^k} f)^2)^{1/2} \\
&\lesssim
\sup_i
(\sum_j \sum_{k}  2^{-\abs{j-k}/q'}( M M_{q,V} A_{u_i}P_{2^k} f)^2)^{1/2} \\
&\lesssim
\sup_i
(\sum_{t} (M M_{q,V} A_{u_i}P_{2^k} f)^2)^{1/2} \\
&\le
(\sum_{k} (M M_{q,V} \sup_i \abs{A_{u_i}P_{2^k} f})^2)^{1/2},
\end{align*}
where $M_{q,V}$ is the $q$-maximal operator in the vertical direction $M_{q,V}f=(M_{V}(f^{q}))^{1/q}$ for any fixed $1<q<2$ with $M_{V}$ as in \eqref{eq:MV}, $M$ is the usual two-dimensional Hardy--Littlewood maximal operator, and the pointwise estimate $\abs{\Delta_j P_{2^{k}}f} \lesssim 2^{\abs{j-k}/q'} M M_{q,V} f$ follows from \cite[Sublemma 4.2]{MR2249617} applied in the vertical direction.
The Fefferman--Stein maximal inequalities and Theorem~\ref{thm:single-scale} give
\[
\norm{ (\sum_{t\in 2^{\Z}} (M M_{q,V} \sup_i A_{u_i}P_t f)^2)^{1/2} }_p
\le
C \norm{ (\sum_{t\in 2^{\Z}} (\sup_i A_{u_i}P_t f)^2)^{1/2} }_p
\le
C \norm{ f }_p.
\]
With Tchebysheff we obtain
\[
\eqref{thirdcww}
=
\meas{\Set{ \sup_i \Delta A_{u_i} f(z)> \epsilon\lambda  }}
\le
C(\epsilon\lambda)^{-p}\norm{f}_p^p
\le
C \log(N+2)^{p/2} \lambda^{-p}\norm{f}_p^p,
\]
and this concludes the proof of Corollary \ref{cor:single-scale-N-directions}.
\end{proof}

\appendix
\section{Lacey--Li covering argument}
\label{sec:LL}
Lacey and Li \cite{MR2654385} have introduced a certain family of  maximal operators associated to a vector field $u$, which they called the ``Lipschitz--Kakeya'' maximal operator:
\[
f\mapsto \sup_{R\in \mathcal R_\delta} \langle f,\mathbf{1}_R\rangle \frac{\mathbf{1}_R}{|R|}
\]
where, using the notation from Section~\ref{sec:single-scale:mass-embed},  $\mathcal R_\delta$ is the collection of those parallelograms  $R$ with $|E(R)|\geq \delta |R|$; that is, the vector field $u$ points within the uncertainty interval of $R$ on (at least a) $\delta$-portion of $R$.
These authors proved that such maximal operators have weak type $(2,2)$ operator norm $O(\delta^{-1/2})$ if the vector field is Lipschitz.
In the same paper, they have further showed that an $L^{p}$ bound for this operator for any $p<2$ implies the $L^{2}$ estimate for the single band version of the directional Hilbert transform.
Bateman and Thiele \cite{MR3148061} gave a streamlined proof of the weak type $(2,2)$ estimate for this maximal operator in the case of a one-variable vector field and used it to obtain square function estimates of the type \eqref{eq:sq-fct} for the directional Hilbert transform.

In this section we further simplify the proof of the weak type $(2,2)$ estimate for this maximal operator, also taking care of Lipschitz vector fields.
We use the notation from Section~\ref{sec:single-scale:mass-embed} and write $L(R)=\meas{I(R)}$.
The main part of the proof is the following covering argument.
\begin{theorem}\label{laceylict}
Let $0< \delta \le 1$ and let $\calR$  be a finite collection of parallelograms with vertical edges and dyadic shadow such that for each $R\in \calR$ we have
\[
|E(R)|\ge \delta |R|
\]
and $L(R)\norm{v}_{\Lip}\leq 1/30$.
Then there is a subset $\calG\subset \calR$ such that
\begin{align} \label{coverct}
|\bigcup_{R\in \calR} R| & \lesssim \sum_{R\in \calG} |R|\ ,
\\
\label{twoonect}
\int (\sum_{R\in \calG}\one_R)^2 & \lesssim \delta^{-1} \sum_{R\in \calG} |R|\ .
\end{align}
\end{theorem}

The set $\calG$ is constructed as in Lemma~\ref{lem:bt-cf-covering}, so that \eqref{coverct} holds by construction.
In the remaining part of this section we will show \eqref{twoonect}.
Expanding the square on the left-hand side of \eqref{twoonect} and using symmetry we obtain the estimate
\[
\sum_{R\in\calG} \meas{R}
+ 2 \sum_{(R,R')\in\calP} \meas{R\cap R'},
\]
where $\calP$ is the set of pairs $(R,R')\in\calG^{2}$ such that $R\cap R'\neq\emptyset$ and $R$ has been chosen before $R'$.
The former term is clearly bounded by the right-hand side of \eqref{twoonect}.
In the latter term we notice first that by \eqref{eq:adm2} we have
\[
\sum_{R'\in\calG} \sum_{R\in\calG(R')} \meas{R\cap R'}
\leq
\sum_{R'\in\calG} \sum_{R\in\calG(R')} L(R') H(R)
\leq
\sum_{R'\in\calG} L(R') H(R'),
\]
and this is also bounded by the right-hand side of \eqref{twoonect}.
Hence it suffices to estimate
\begin{equation}
\label{eq:ll:PR}
\sum_{R\in \calG}\sum_{R'\in\calP(R)} \meas{R\cap R'},
\end{equation}
where
\[
\calP(R) := \Set{R' : (R,R')\in \calP, U(R)\cap 10 U(R') = \emptyset}.
\]

First we clarify the position of $U(R')$ relative to $U(R)$ when $R'\in\calP(R)$.

\begin{lemma}
\label{lem:lk:lac}
Suppose $R'\in\calP(R)$.
Then
\[
\max (|U(R)|,|U(R')|) \leq \frac14 \dist(U(R'),U(R)).
\]
\end{lemma}
\begin{proof}
We distinguish two cases: 
\begin{enumerate}
\item $|U(R)|\le |U(R')|$.
In this case we use the definition of $\calP(R)$.
\item $|U(R)|>|U(R')|$.
In this case we have
\[
H(R') = |U(R')| L(R') < |U(R)| L(R) = H(R),
\]
and in particular $7H(R') \leq H(10R)$.
If the conclusion was false, then $10U(R) \cap U(R')\neq\emptyset$, and by Lemma~\ref{7rlemma} we obtain $7R'\subset 70R$.
This contradicts the hypothesis that $R'$ was added to $\calG$ after $R$.
\end{enumerate}
\end{proof}

The next lemma gives a condition for two parallelograms to have comparable slopes.
This is the only place where the Lipschitz hypothesis is used.
Denote the projection onto the first coordinate by $\Pi$.
\begin{lemma}
\label{lem:lk:near}
Assume $L(R) \norm{v}_{\Lip} \leq 1/30$.
Suppose $R',R''\in\calP(R)$ and $\Pi E(R') \cap \Pi E(R'') \neq \emptyset$.
Then
\[
\dist(U(R'),U(R'')) \leq \frac{1}{8} \dist(U(R),U(R')).
\]
\end{lemma}
\begin{proof}
Let $x\in \Pi E(R') \cap \Pi E(R'')$.
The distance of the points $y',y''$ such that $(x,y')\in R'$ and $(x,y'')\in R''$ is bounded above by
\[
H(R) + H(R') + H(R'') + L(R') \dist(U(R),U(R')) + L(R'') \dist(U(R),U(R'')).
\]
Choosing $(x,y')\in E(R')$ and $(x,y'')\in E(R'')$ and using the Lipschitz hypothesis and Lemma~\ref{lem:lk:lac} we obtain
\begin{multline*}
\dist(U(R'),U(R''))
\leq
\norm{v}_{\Lip}\Big(H(R) + H(R') + H(R'')\\
+ L(R') \dist(U(R),U(R')) + L(R'') \dist(U(R),U(R''))\Big)\\
\leq
\frac{1}{30} \Big(|U(R)| + |U(R')| + |U(R'')|+ \dist(U(R),U(R')) + \dist(U(R),U(R''))\Big)\\
\leq
\frac{1}{30} \Big(\frac64 \dist(U(R),U(R')) + \frac54 \dist(U(R),U(R''))\Big)\\
\leq
\frac{1}{30} \Big(\frac{11}{4} \dist(U(R),U(R')) + \frac54 |U(R')| + \frac54 \dist(U(R'),U(R''))\Big)\\
\leq
\frac{1}{30} \Big(\frac{13}{4} \dist(U(R),U(R')) + \frac54 \dist(U(R'),U(R''))\Big).
\end{multline*}
The conclusion follows.
\end{proof}

The basic estimate for the size of the intersection of two parallelograms is the size of the intersection of infinite stripes containing them:
\begin{lemma}
Let $R,R' \in \calR$.
Then
\begin{equation}\label{seconduict}
|R\cap R'| \leq \dist(U(R),U(R'))^{-1} H(R) H(R').
\end{equation}
\end{lemma}
\begin{proof}
By a shearing transformation we may assume that the central line segment of $R$ is horizontal.
Let $u_{0}$ be the central slope of $R'$.
Then $R\cap R'$ is contained in a parallelogram of height $H(R)$ and base $H(R')u_0^{-1}$.
On the other hand, $u_{0}\geq\dist(U(R),U(R'))$.
\end{proof}

We decompose the set $\calP(R)$ dyadically according to the distance between $U(R)$ and $U(R')$.
Specifically, for $k\in\N$ let
\[
\calP_{k}(R) := \Set{ R'\in\calP(R) : 2^{k-3} < \frac{\dist(U(R),U(R'))}{|U(R)|} \leq 2^{k-2}}.
\]
For a fixed $k$ we will estimate the contribution of $\calP_{k}(R)$ to \eqref{eq:ll:PR} using a stopping time argument.
For a dyadic interval $I$ denote $R_{I}  := R \cap (I\times\R)$.
\begin{lemma}\label{nosmallct}
Let $I \subseteq I_{R}$ be a dyadic interval such that there exists $R''\in \calP_{k}(R)$ with $I_{R''}\subseteq I$.
Then
\[
\sum_{R'\in \calP_{k}(R): I\subseteq I_{R'}} |R_I\cap R'| \leq 2 |R_I|.
\]
\end{lemma}

\begin{proof}
Let $R''\in \calP_{k}(R)$ be the parallelogram with $I_{R''}\subseteq I$ that has been chosen last.
Let
\[
\calQ := \Set{R'\in \calP_{k}(R): I\subseteq I_{R'}, R_{I}\cap R'\neq\emptyset} \setminus \Set{R''}.
\]
Since $|R_{I}\cap R''| \leq |R_{I}|$, it suffices to show
\begin{equation}
\label{eq:nosmallct:r}
\sum_{R'\in \calQ} |R_I\cap R'| \leq 10^{-1} |R_I|.
\end{equation}
Assume for contradiction that \eqref{eq:nosmallct:r} fails.
Let $U:= 2^{k} U(R)$.
By Lemma~\ref{lem:lk:lac} we have $U(R'')\subset U$ and thus
\[
H(R'')
\leq
|U| |I_{R''}|
\leq
|U| |I|.
\]
In particular
\[
R''  \subset 50  (1+|U| |I|/H(R)) R_{I} =: \tilde R.
\]
The parallelogram $R''$ has been selected for $\calG$ after the parallelogram $R$ and the parallelograms $R'\in \calQ$.
To obtain a contradiction with the construction of $\calG$ it suffices to show that  
\[
M_V(\one_R + \sum_{R'\in \calQ} \one_{R'})
\] where $M_V$ is the vertical directional maximal function,
is larger than $10^{-3}$ on the parallelogram $\tilde R$.

First assume there exists $R'\in \calQ$ with $H(R')\ge 20|U||I|$.
Note that
\[
U(R')\subset U \subset U(\tilde{R}).
\]
Applying Lemma \ref{7rlemma} to the rectangles $R'_I$ and $\tilde{R}$ we obtain
\[
M_V (\one_{R'}+\one_R)
\geq
7^{-1}H(\tilde{R})^{-1} \big( \min(H(R'),H(\tilde{R}))+H(R) \big)
>
10^{-3}
\]
on $\tilde{R}$, which proves Lemma \ref{nosmallct} in the given case.

Hence we may assume 
\[
H(R')\le 20 |U||I|
\]
for every $R'\in \calQ$.
We then have on $\tilde{R}$ that
\begin{align*}
\MoveEqLeft
M_V (\one_R   + \sum_{R'\in \calQ}\one_{R'})\\
&\ge
H(\tilde{R})^{-1} (H(R)+\sum_{R'\in\calQ} H(R'))\\
&\ge
H(\tilde{R})^{-1} (H(R)+\sum_{R'\in\calQ} |R_I\cap R'| |U| H(R)^{-1})
&\text{by \eqref{seconduict}}\\
&\ge
H(\tilde{R})^{-1} (H(R)+|U| H(R)^{-1} 10^{-1} |R_I|)
&\text{since \eqref{eq:nosmallct:r} fails}\\
&\ge
500^{-1}.
\end{align*}
This completes the proof of Lemma \ref{nosmallct}.
\end{proof}

\begin{corollary}
\label{cor:lk:1scale}
\[
\sum_{R'\in \calP_{k}(R)} |R\cap R'|
\leq
4 H(R) \cdot \meas[\big]{\bigcup_{R'\in \calP_{k}(R)} \Pi(R')}.
\]
\end{corollary}

\begin{proof}
Let $\calI$ be the set of maximal dyadic intervals contained in $\cup_{R'\in \calP_{k}(R)} \Pi(R')$ that do not contain $I_{R'}$ for any $R'\in\calP_{k}(R)$.
For each $I\in\calI$ let $\tilde I$ denote its dyadic parent.
Then by maximality of $I$ and Lemma \ref{nosmallct} we have
\begin{align*}
\sum_{R'\in \calP_{k}(R)} |R_I\cap R'|
&=
\sum_{R'\in \calP_{k}(R) : I\subsetneq I_{R'}} |R_I\cap R'|\\
&\leq
\sum_{R'\in \calP_{k}(R) : \tilde I\subseteq I_{R'}} |R_{\tilde I}\cap R'|\\
&\leq
2 |R_{\tilde I}|
\leq
4 |R_{I}|.
\end{align*}
The set $\calI$ is a covering of $\cup_{R'\in \calP_{k}(R)} \Pi(R')$, so the conclusion of the lemma follows after summing over all intervals in $\calI$.
\end{proof}

We are now in position to complete the proof of Theorem~\ref{laceylict} by estimating \eqref{eq:ll:PR}:
\begin{align*}
\sum_{R'\in\calP(R)} |R\cap R'|
&=
\sum_{k\in\N} \sum_{R'\in \calP_{k}(R)} |R\cap R'|\\
&\lesssim
H(R) \sum_{k\in\N} |\cup_{R'\in \calP_{k}(R)} \Pi(R')|
&\text{by Corollary~\ref{cor:lk:1scale}}\\
&=
H(R) \sum_{k\in\N} \sum_{R'\in \calP_{k}'(R)} |\Pi(R')|\\
&\lesssim
\delta^{-1} H(R) \sum_{k\in\N} \sum_{R'\in \calP_{k}'(R)} |\Pi E(R')|\\
&\lesssim
\delta^{-1} |R|,
\end{align*}
where $\calP_{k}'(R)\subset\calP_{k}(R)$ is a system of representatives for maximal intervals $I_{R'}$, in the penultimate step we have used the density hypothesis in the form $|\Pi(R')| \leq |\Pi E(R')|/\delta$, and in the last step we have used Lemma~\ref{lem:lk:near} to conclude that the projections there have bounded overlap.

\printbibliography
\end{document}
